\documentclass[10pt]{amsart}
\usepackage{amsmath,latexsym,amssymb,amsfonts}
\usepackage{amscd,amssymb,epsfig}
\usepackage{hyperref}
\usepackage{mathdots}

\newtheorem{theorem}{Theorem}[section]
\newtheorem{lemma}[theorem]{Lemma}

\newtheorem{proposition}[theorem]{Proposition}
\newtheorem{corollary}[theorem]{Corollary}
\newtheorem{definition}[theorem]{Definition}

\newtheorem{problem}{Problem}
\theoremstyle{remark}
\newtheorem{example}[theorem]{Example}

\newcommand{\A}{\mathcal{A}}
\newcommand{\C}{\mathbb{C}}
\newcommand{\E}{\mathbb{E}}
\newcommand{\N}{\mathbb{N}}
\newcommand{\R}{\mathbb{R}}

\renewcommand{\H}{\mathbb{H}}

\newcommand{\id}{\operatorname{id}}

\newcommand{\sa}{\operatorname{sa}}

\title[Subordination and fixed points]{Analytic subordination theory of operator-valued free additive convolution and the solution of a general random matrix problem}

\author[S. Belinschi]{Serban T. Belinschi}
\address{Department of Mathematics \& Statistics,
Queen's University, and Institute of Mathematics ``Simion
Stoilow'' of the Romanian Academy;
Department of Mathematics and Statistics, 
Queen's University,
Jeffrey Hall, 
Kingston, ON K7L 3N6 CANADA} 
\email{sbelinsch@mast.queensu.ca}

\author[T. Mai]{Tobias Mai}
\address{Universit\"{a}t des Saarlandes, FR $6.1-$Mathematik, 66123 Saarbr\"{u}cken, Germany } 
\email{mai@math.uni-sb.de}

\author[R. Speicher]{Roland Speicher}
\address{Universit\"{a}t des Saarlandes, FR $6.1-$Mathematik, 66123 Saarbr\"{u}cken, Germany } 
\email{speicher@math.uni-sb.de}

\date{\today}
\begin{document}

\thanks{Work of S.T. Belinschi was supported by a Discovery Grant from the Natural Sciences and Engineering Research Council of Canada. STB also gratefully acknowledges the support of the Alexander von Humboldt
foundation through a Humboldt fellowship for experienced researchers and the hospitality and great work environment offered by the Free Probability workgroup at the Universit\"{a}t des Saarlandes during most of the work on this paper. \newline\indent
Work of T. Mai and R. Speicher were supported by funds from the Alfried Krupp von Bohlen und 
Halbach Stiftung ("R\"uckkehr deutscher Wissenschaftler aus dem Ausland") and from the DFG (SP-419-9/1).}

\begin{abstract}
We develop an analytic theory of operator-valued additive free convolution
in terms of subordination functions. In  contrast to earlier investigations our functions are not just given by power series expansions, but are defined as Frechet analytic functions in all of the operator upper half plane. Furthermore, we do not have to assume that
our state is tracial. Combining this new analytic theory of operator-valued free convolution with Anderson's selfadjoint version of the linearization trick we are able to provide a solution to the following general random matrix problem:
Let $X_1^{(N)},\dots, X_n^{(N)}$ be selfadjoint $N\times N$ random matrices which are, for
$N\to\infty$, asymptotically free. Consider a selfadjoint polynomial $p$ in $n$ non-commuting variables and let $P^{(N)}$ be the element $P^{(N)}=p(X_1^{(N)},\dots,X_n^{(N)})$. How can we calculate the asymptotic eigenvalue distribution of $P^{(N)}$ out of the asymptotic eigenvalue distributions of $X_1^{(N)},\dots,X_n^{(N)}$?
\end{abstract}

\maketitle

\section{Introduction}

Free probability theory \cite{VDN,HP,NSbook} 
originated from questions about the structure of 
von Neumann algebras related to free groups, but has evolved into
a subject with links to many other, a priori quite unrelated, fields: most notably random matrices, but also combinatorics, representation theory of large groups, mathematical physics, or as applied fields as financial correlations and wireless communications.

The central notion in the theory is the one of ``free independence'' or ``freeness''. This replaces in the context of non-commuting variables the classical notion of independence. Whereas the latter is based on tensor products, free independence is modelled according to free products. 

The relation between free probability theory and random matrices relies on the fundamental insight of Voiculescu \cite{V-inventiones} that typical independent random matrices become asymptotically free
in the large $N$ limit. This triggered a surge of
new results on operator algebras as well as on random matrices.
Based on this, free probability techniques have become prominent in wireless communications \cite{TV} or in computational methods for random matrices \cite{RE}.

In this paper we will provide a solution to the following basic problem in free probability theory.

\begin{problem}
Let $x_1,\dots, x_n$ be selfadjoint elements which are freely independent. Consider a selfadjoint polynomial $p$ in $n$ non-commuting variables and let $P$ be the element $P=p(x_1,\dots,x_n)$. How can we calculate the distribution of $P$ out of the distributions of $x_1,\dots,x_n$?
\end{problem}

By Voiculescu's asymptotic freeness results typical random matrices (like deterministic and independent Wigner, Wishart, or Haar unitary matrices) become asymptotically free
in the large $N$ limit. This means that the asymptotic eigenvalue distribution of a random matrix, which is given as some polynomial in such asymptotically free random matrices, is
given by the distribution of the corresponding polynomial in free variables.
Thus, if we can deal with polynomials in free random variables we can deal with any random matrix of this type.
Hence our result yields as a direct consequence also a solution to the following random matrix problem.

\begin{problem}
Let $X_1^{(N)},\dots, X_n^{(N)}$ be selfadjoint $N\times N$ random matrices which are, for
$N\to\infty$, asymptotically free. Consider a selfadjoint polynomial $p$ in $n$ non-commuting variables and let $P^{(N)}$ be the element $P^{(N)}=p(X_1^{(N)},\dots,X_n^{(N)})$. How can we calculate the asymptotic eigenvalue distribution of $P^{(N)}$ out of the asymptotic eigenvalue distributions of $X_1^{(N)},\dots,X_n^{(N)}$?
\end{problem}

For both Problems 1 and 2 there is a long list of contributions which provide 
solutions for special choices of the polynomial $p$. In the context of free probability, Voiculescu solved in \cite{DVJFA} and \cite{V2} Problem 1 for the cases of $p(x,y)=x+y$ and $p(x,y)=xy^2x$
(corresponding to the additive and multiplicative free convolution) with the introduction of the $R$- and $S$-transform, respectively. Nica and
Speicher could give in \cite{NS-commut} a solution for the problem of the free commutator, $p(x,y)=i(xy-yx)$. 
In the random matrix context, Problem 2
was addressed for various polynomials - and usually, also for specific choices of the distributions of the $X_i^{(N)}$ - in work of: Marchenko-Pastur \cite{MP},
Girko \cite{Gir}, Bai and Silverstein \cite{BS}; Tulino and Verdu \cite{TV}; M\"uller \cite{Mue}; Moustakis and Simon \cite{MouS}; Hechem, Loubaton, and Najim \cite{HLN}; Couillet and Debbah \cite{CB}; and many more. For a more extensive list of contributions in this context we refer to the books \cite{BS,TV,CB}. Some of those situations were also treated by operator-valued
free probability tools, see in particular \cite{SpV,BSTV}.
The general case for the anti-commutator in the random matrix context, $p(x,y)=xy+yx$, was treated by Vasilchuk \cite{Vasilchuk}. 

All those investigations where specific for the considered polynomial and up to now there has not
existed a master algorithm which would work for all polynomials. We will provide here such a general algorithm.
Our solution will rely on a combination of the so-called linearization trick with new advances on the analytic theory of operator-valued convolutions; from a technical point of view the latter is the main contribution of the present work.
 
The general philosophy of the linearization trick is that a complicated scalar-valued problem can be transformed into a simpler 
operator-valued problem. 
This idea can be traced back to the early papers \cite{V1995} of Voiculescu, but has become quite prominent in the seminal paper \cite{HaagerupThorbjornsen2} of Haagerup and Thorbj\o rnsen.
More precisely, the
linearization trick from \cite{HaagerupThorbjornsen2} says that for understanding properties of a polynomial $p(x_1,\dots,x_n)$
in some (in general, non-commuting) variables $x_1,\dots,x_n$ it suffices
to consider a linear polynomial $L:=a_1\otimes x_1+\cdots a_n\otimes x_n$.
The price to pay for this transition from a general to a linear polynomial is
that the linear polynomial $L$ has operator-valued coefficients (i.e., the $a_1,
\dots,a_n$ are matrices, of suitably chosen size, depending on the polynomial $p$, but not on the variables $x_1,\dots,x_n$). We will be mainly interested in selfadjoint polynomials $p$. The original approach of Haagerup and Thorbj\o rnsen has the drawback that it does not preserve selfadjointness.
However, recently Anderson \cite{Anderson1} refined the linearization trick in that respect, i.e., he showed that the coefficients $a_i$ can actually be chosen in such a way that $L$ is also selfadjoint.

So the linearization trick gives that in order to understand the distribution of $p(x_1,\dots,x_n)$ it suffices to understand the operator-valued distribution of $L$. However, $L$ is now an operator-valued linear combination of 
$x_1,\dots,x_n$. Since basic properties of freeness imply that the freeness of $x_1,\dots,x_n$ implies the operator-valued freeness of $a_1\otimes x_1, \dots, a_n\otimes x_n$ we get in our situation that the distribution of $L$ is given by the operator-valued
free convolution of its summands; the operator-valued distribution of $a_i\otimes x_i$, however,  is determined by the distribution of $x_i$.
So we have reduced our original non-linear problem to a linear one about operator-valued convolution. 

There exists already a theory by Voiculescu for dealing with such operator-valued convolutions, 
relying on the so-called operator-valued $R$-transform. Up to now, this theory has essentially 
been developed on the level of formal power series, see \cite{V1995,RS2}. It has, however, to be remarked 
that we need a good analytic description of the operator-valued convolution in order to arrive at a useful 
description of the solution of our problem. Since there is rarely a situation where explicit solutions of those 
equations are available (and that holds true for both the scalar-valued and the operator-valued case), 
it is important that the equations are given in a form which are analytically controllable and 
numerically implementable. 
In the scalar-valued case is has become more and more apparent in the last years (see, e.g., \cite{BB07}) 
that using a subordination approach to free convolution (which, as also noted -- and for the first time explicitly
formulated -- in \cite{CG}, is equivalent to the $R$-transform approach, 
but has better analytic properties)
is more suited for analytic questions. 
The subordination property (see statements (1) and (3) of Theorem \ref{main2} below) has first been proved in
\cite{V3} by Voiculescu for scalar-valued free additive convolution, with the purpose of investigating $L^p$-regularity
of convolution measures in the context of free entropy.
In the operator-valued case there exist also already subordination results, by Biane \cite{Biane1} 
and Voiculescu \cite{V2000,FreeMarkov}, 
based on a power-series description of the functions involved and algebraic properties of freeness with 
amalgamation. 
Our main  contribution in this article to the question of subordination will be to 
develop a general analytic theory of subordination for general operator-valued convolutions and to show that the 
fixed point equations arising from this subordination formulation are indeed analytically 
controllable and numerically easily implementable. 

We also want to point out that the previous works of Biane and Voiculescu
were restricted to a tracial frame; in our general analytic approach, however,
we dot need to assume that our conditional expectation is given with respect to a tracial state, or for 
that matter that such a state at all exists on the given algebra.

Having reduced the problem of polynomials in free variables to operator-valued free additive convolution it is conceivable that further progress
on analytic properties of operator-valued convolution will
result in understanding qualitative properties of arbitrary selfadjoint polynomials in free
variables. This will be pursued in future investigations. In this context we want
to point out that the absence of atoms for the distribution of selfadjoint polynomials in free semicirculars (or more generally, in free variables without
atoms) has been established recently by Shlyakhtenko and Skoufranis \cite{ShSk}, by
quite different methods. Another direction for future work is the extension of
the present approach to general, not necessarily selfadjoint polynomials (where the eigenvalue distribution will be replaced by the Brown measure). This will be pursued in \cite{BSS}.

The content of the article is as follows. In Section 2, we develop the analytic subordination description of operator-valued convolutions. Theorem \ref{main2} is the main result about this. In Section 3, we will recall the linearization trick, in the form as given by Anderson, and streamlined to our needs. In Section 4, we will combine the linearization trick with our description of operator-valued convolution and show how to solve Problem 1 (and thus also Problem 2) for general selfadjoint polynomials $p$. In Section 5, we will present some examples, by applying our algorithm to special polynomials and also compare this with numerical simulations of eigenvalue distributions for corresponding random matrices.

\section{The $R$-transform and subordination in the operator-valued context}

Following Voiculescu \cite{V1995},
we will call an {\em operator-valued non-commutative probability space}
a triple $(\mathcal M,\mathbb E,B)$, where $\mathcal M$ is a unital Banach 
algebra, $B\subseteq\mathcal M$ is a Banach subalgebra
containing the unit of $\mathcal M$, and $\mathbb E\colon\mathcal M
\to B$ is a unit-preserving conditional expectation. In most cases, $\mathcal M$ 
will be a von Neumann algebra, $B$ a $W^*$-subalgebra of $\mathcal M$ and thus
$\mathbb E$ will be completely positive weakly continuous.
Elements
in $\mathcal M$ will be called {\em operator-valued (or $B$-valued)
random variables.} The {\em distribution} of a random variable
$x\in\mathcal M$ with respect to $\mathbb E$ is, by definition,
the set of multilinear maps
$$
\mu_x:=\{m_n^x\colon B^{n-1}\to B\colon m_n(b_1,\dots,b_{n-1})=
\mathbb E[xb_1xb_2\cdots xb_{n-1}x], n\in\mathbb N\}.
$$
We call $m_n^x$ the \emph{$n^{\rm th}$ moment} of $x$ (or, equivalently,
of $\mu_x$).
It will be convenient to interpret $m_0^x$ as the constant
equal to ${\bf1}$, the unit of $B$ (or, equivalently, of $\mathcal M$)
and $m_1^x=\mathbb E[x]$, the expectation of $x$.
We denote by $B\langle x_1,\dots,x_n\rangle$ the algebra
generated by $B$ and the elements $x_1,\dots,x_n$ in $\mathcal M$.

\begin{definition}\label{B-freeness}
Two algebras $A_1,A_2\subseteq\mathcal M$ containing $B$ are called 
{\em free with amalgamation over} $B$ with respect to $\mathbb E$
(or just {\em free over} $B$) if
$$
\mathbb E[x_1x_2\cdots x_n]=0
$$
whenever $n\in\mathbb N$, $x_j\in A_{i_j}$ satisfy $\mathbb E[x_j]=0$ for all $j$
and $i_j\neq i_{j+1}$, $1\leq j\leq n-1$. Two random variables $x,y\in
\mathcal M$ are called free over $B$ if $B\langle x\rangle$ and
$B\langle y\rangle$ are free over $B$.
\end{definition}
If $x,y\in\mathcal M$ are free over $B$, then $\mu_{x+y}$  depends only 
on $\mu_x$ and $\mu_y$. Following Voiculescu, we shall denote this
dependency by $\mu_x\boxplus\mu_y$, 
and call it the \emph{free additive
convolution} of the distributions $\mu_x$ and $\mu_y$.
It is known \cite{RS2,V1995} that $\boxplus$ is commutative and associative.

A very powerful tool for the study of operator-valued 
distributions is the generalized Cauchy-Stieltjes transform and
its fully matricial extension \cite{V1995,V2000}: for a fixed
$x\in\mathcal M$, we define $G_x(b)=\mathbb E\left[(b-x)^{-1}\right]$
for all $b\in B$ for which $b-x$ is invertible in $\mathcal M$. One can 
easily verify that $G_x$ is a holomorphic mapping on an open subset of 
$B$. Its {\em fully matricial extension} $G_x^{(n)}$ is defined on
the set of elements $b\in M_n(B)$ for which $b-x\otimes 1_n$ is invertible in $M_n(\mathcal M)$, by the relation $G_x^{(n)}(b)=
\mathbb E\otimes\text{Id}_{M_n(\mathbb C)}\left[(b-x\otimes 1_n)^{-1}
\right]$. It is a crucial observation of Voiculescu that 
the family $\{G_x^{(n)}\}_{n\ge1}$ encodes the distribution
$\mu_x$ of $x$. A succinct description of how to identify
the $n^{\rm th}$ moment of $x$ when $\{G_x^{(n)}\}_{n\ge1}$ is
known is given in \cite{BPV2010}.

Up to this point, all the notions we introduced required no involutive
structure and no positivity on either $\mathcal M$ or $B$. In particular,
the definition of $G_x$ only requires the existence of a multiplicative 
structure on $\mathcal M$ so that $\|xy\|\leq M\|x\|\|y\|$ for some
$M>0$ independent of $x$ and $y$ (although a certain - bounded - dependence 
on $\|x\|$ and $\|y\|$ could obviously be allowed). The reader can 
verify for herself that this holds also for most of the analytic transforms
which we introduce below. From now on, we will assume that $\mathcal M$ is in addition a
$C^*$-algebra, $B$ is a $C^*$-subalgebra of $\mathcal M$, and $\mathbb E$ is completely
positive, i.e., that $\mathcal M$ is a \emph{$C^*$-operator-valued non-commutative probability space}.

In the following we will use the notation $x>0$ for the situation where
$x\geq 0$ and $x$ is invertible; note that this is equivalent to the fact
that there exists a real $\varepsilon>0$ such that $x\geq \varepsilon{\bf1}$. From the latter
it is clear that $x>0$ implies $\mathbb E[x]>0$ (because our conditional expectations are 
automatically completely positive). Any element $x\in\mathcal M$
can be uniquely written as $x=\Re x+i\Im x$, where $\Re x=\frac{x+x^*}{2}$
and $\Im x=\frac{x-x^*}{2i}$ are selfadjoint. We call $\Re x$ and $\Im x$
the real and imaginary part of $x$.

From now on we shall restrict our attention to the case when
$x,y$ are selfadjoint. 
In this case, one of appropriate domains for $G_x$ - and the 
domain we will use most - is the {\em operator upper half-plane}
$\mathbb H^+(B):=\{b\in B\colon\Im b>0\}.$ Elements in this open set
are all invertible, and $\mathbb H^+(B)$ is invariant under conjugation
by invertible elements in $B$, i.e. if $b\in\mathbb H^+(B)$ and $c\in GL(B)$ is invertible,
then $cbc^*\in\mathbb H^+(B)$. It has been noted
in \cite{V2000} that $G_x^{(n)}$ maps $\mathbb H^+(M_n(B))$ into
the operator lower half-plane $\mathbb H^-(M_n(B)):=-
\mathbb H^+(M_n(B))$ and has ``good behaviour at infinity'' in the
sense that $\displaystyle\lim_{\|b^{-1}\|\to0}bG_x^{(n)}(b)=
\lim_{\|b^{-1}\|\to0}G_x^{(n)}(b)b={\bf1}$. 

As, from the analytic perspective, $G_x^{(n)}$ have 
essentially the same behaviour on $\mathbb H^+(M_n(B))$ for any
$n\in\mathbb N$, we shall restrict our analysis from now on to 
$G_x=G_x^{(1)}$. However, all properties we deduce for this
$G_x$, and all the related functions we shall introduce, remain
true, under the appropriate formulation, for all $n\ge1$.

We shall use the following analytic mappings, all defined on
$\mathbb H^+(B)$; all transforms have a natural
Schwarz-type analytic extension to the lower half-plane given by
$f(b^*)=f(b)^*$;
in all formulas below, $x=x^*$ is fixed in $\mathcal M$:
\begin{itemize}
\item the moment generating function:
\begin{equation}\label{psi}
\Psi_x(b)=\mathbb E\left[({\bf1}-bx)^{-1}-{\bf1}\right]=\mathbb E\left[
(b^{-1}-x)^{-1}\right]b^{-1}-{\bf1}=G_x(b^{-1})b^{-1}-{\bf1};
\end{equation}
\item The reciprocal Cauchy transform:
\begin{equation}\label{F}
F_x(b)=\mathbb E\left[(b-x)^{-1}\right]^{-1}=G_x(b)^{-1};
\end{equation}
\item In this paper we shall call the following function \emph{``the h transform:''}
\begin{equation}\label{h}
h_x(b)=\mathbb E\left[(b-x)^{-1}\right]^{-1}-b=F_x(b)-b;
\end{equation}
\end{itemize}
It has been shown in \cite{BPV2010} that 
\begin{equation}\label{ImF}
\Im F_x(b)\ge\Im b,\quad b\in\mathbb H^+(B),
\end{equation}
and thus
\begin{equation}\label{Imh}
h_x(\mathbb H^+(B))\subseteq\overline{\mathbb H^+(B)}.
\end{equation}

Based on the moment generating function, Voiculescu \cite{V1995} introduced
the operator-valued $R$-transform. It has the fundamental property that
it linearizes free additive convolution of operator-valued distributions: if 
$x$ and $y$ are free over $B$, then 
\begin{equation}\label{R-lineariz}
R_x(b)+R_y(b)=R_{x+y}(b),\quad\text{for }b\text{ in some neighbourhood of }0.
\end{equation}
Equivalently, in terms of distributions, the above equation is written as
$$
R_{\mu_x}(b)+R_{\mu_y}(b)=R_{\mu_x\boxplus\mu_y}(b).
$$
We will use here the (equivalent) definition for the $R$-transform provided by 
Dykema in \cite{Str}. Define
\begin{equation}\label{C}
\mathcal C_x(b)=\mathbb E\left[({\bf1}-bx)^{-1}\right]b=\sum_{n=0}^\infty
\mathbb E\left[(bx)^n\right]b,\quad \|b\|<\|x\|^{-1}.
\end{equation}
Then \cite[Proposition 4.1]{Str}, since $\mathcal C_x'(0)=\text{Id}_B$, by the 
inverse function theorem there exists a unique $B$-valued analytic map
$R_x$ defined on a neighbourhood of zero so that 
\begin{equation}\label{C-R}
\mathcal C_x^{\langle-1\rangle}(b)=({\bf1}+bR_x(b))^{-1}b=b({\bf1}+R_x(b)b)^{-1},\quad\|b\|\text{ small enough}.
\end{equation}
($\mathcal C_x^{\langle-1\rangle}$ denotes the compositional inverse of $\mathcal C_x$ on
some small enough neighbourhood of zero.) Then \cite[Theorem 4.7]{Str} equation
\eqref{R-lineariz} takes place.

For our purposes, it is important to note that $\mathcal C_x(b^{-1})=G_x(b)$,
$b\in\mathbb H^+(B)$. Then we have formally $G_x^{\langle-1\rangle}(b)=(\mathcal C_x^{\langle-1\rangle}(b))^{-1}
=b^{-1}({\bf1}+bR_x(b))=({\bf1}+R_x(b)b)b^{-1},$ so that
\begin{equation}\label{R-def}
R_x(b)=G_x^{\langle-1\rangle}(b)-b^{-1},
\end{equation}
as in the scalar case \cite{DVJFA}. To make this rigorous: there exists an $r=r_x>0$ so that both
$\mathcal C_x$ and $\mathcal C^{\langle-1\rangle}_x$ are defined on $B(0,r)=\{b\in B\colon\|b\|<r\}$
and map $B(0,r/2)$ strictly inside $B(0,r)$. The set 
$\{b\in\mathbb H^+(B)\colon\|b^{-1}\|<r/2\}$ is open and nonempty, since it
contains $b_0=\frac{4i}{r}{\bf1}$. $G_x(w)=\mathcal C_x(w^{-1})$ 
will then necessarily map this set into $B(0,r)$. 
Observe that 
$$G_x(b_0)=\mathbb E[(\frac{4i}{r}{\bf1}-x)^{-1}]=-\frac{4i}{r}
\mathbb E[(\frac{16}{r^2}{\bf1}+x^2)^{-1}]-\mathbb E[(\frac{16}{r^2}{\bf1}+x^2)^{-1}x]\in\mathbb H^-(B)$$
is trivially invertible and $\|G_x(b_0)\|<\frac{r}{4}+r^2\frac{\|x\|}{16}<\frac{r}{2}$ whenever
$r<\|x\|/4$. 
By the continuity of the inverse function and the openness
of the set of invertible elements, we have shown that there exists an open set
$V_x$ included in $\mathbb H^-(B)\cap B(0,r)$ on which $G_x^{\langle-1\rangle}$ is defined
and $G_x^{\langle-1\rangle}(b)=(\mathcal C_x^{-1}(b))^{-1}.$
Thus, on an open set included in $B(0,r)$, equation \eqref{R-def} takes place.
Since $R_x$ is uniquely determined on $B(0,r)$, the identity principle for analytic maps
on Banach spaces guarantees that \eqref{R-def} is enough to determine $R_x$, and
knowledge of one of $R$ or $G$ (on the corresponding domains) implies knowledge of 
the other.

In \cite{V2000}, under certain assumptions on $(\mathcal M,\mathbb E,B)$, Voiculescu
proved the existence of two fully matricial self-maps $\omega_1,\omega_2$ of 
$\mathbb H^+(B)$ having good asymptotics ``at infinity'' which satisfy
the relations $G_x\circ\omega_1=G_y\circ\omega_2=G_{x+y}$ on $\mathbb H^+(B)$.
Thus, locally,
$\omega_1=G_x^{\langle-1\rangle}\circ G_{x+y}$ and $\omega_2=G_y^{\langle-1\rangle}\circ G_{x+y}$. Provided one can show that the functions involved
are well-defined on the set $\{b\in\mathbb H^+(B)\colon\|b^{-1}\|<r/2\}$ for some $r>0,$
equations \eqref{R-lineariz} and \eqref{R-def} provide us with the relation
\begin{eqnarray}
\omega_1(b)+\omega_2(b)&= & G_x^{\langle-1\rangle}(G_{x+y}(b))+G_y^{\langle-1\rangle}(G_{x+y}(b))\nonumber\\
& = & R_{x}(G_{x+y}(b))+R_y(G_{x+y}(b))+2G_{x+y}(b)^{-1}\nonumber\\
& = & (R_{x+y}(G_{x+y}(b))+G_{x+y}(b)^{-1})+G_{x+y}(b)^{-1}\nonumber\\
& = & b+F_{x+y}(b),\label{Bercovici-Voiculescu}
\end{eqnarray}
for all $b\in\{b\in\mathbb H^+(B)\colon\|b^{-1}\|<r/2\}$. We have used \eqref{R-def}
in the second and last equalities and \eqref{R-lineariz} in the third.

Equation
\eqref{Bercovici-Voiculescu} appeared first in \cite[Lemma 7.2]{BV-reg} in the scalar
case. The non-triviality of our argument will consist in ``following the domains.''
Regrettably, this non-triviality haunts us: we are not able to conclude directly from the 
above that $\omega_j(b)\in\mathbb H^+(B)$, and thus, a priori, we are not able to 
write $G_x(\omega_1(b))=G_y(\omega_2(b))=G_{x+y}(b)$ for all 
$b\in\{b\in\mathbb H^+(B)\colon\|b^{-1}\|<r/2\}$. To avoid this problem (and thus give a rigorous definition of $\omega_1$ and $\omega_2$) we shall
switch back to $\mathcal C$: the function $\omega_1$ can be re-written as
\begin{equation}\label{omega-C}
\left(\omega_1(b^{-1})\right)^{-1}=\left(G_x^{\langle-1\rangle}(G_{x+y}(b^{-1}))\right)^{-1}=
\mathcal C_x^{\langle-1\rangle}(\mathcal C_{x+y}(b)),\quad\|b\|<r/2.
\end{equation}
Thus, $b\mapsto\left(\omega_1(b^{-1})\right)^{-1}$ has an analytic extension
to a neighbourhood of zero which fixes zero. We shall denote this extension by $o_1$
and write 
\begin{equation}\label{o1aroundzero}
\mathcal C_x(o_1(b))=\mathcal C_{x+y}(b),\quad \|b\|\text{ small enough}.
\end{equation}
Similarly
$\mathcal C_y(o_2(b))=\mathcal C_{x+y}(b)$, $\|b\|<r/2$.

Let us now state our main theorem.

\begin{theorem}\label{main2}
Assume that $(\mathcal M, \mathbb E,B)$ is a $C^*$-operator-valued non-commutative 
probability space and $x,y\in\mathcal M$ are two selfadjoint operator-valued random variables
free over $B$. Then there exists a unique pair of Fr\'echet (and thus also G\^{a}teaux) analytic maps $\omega_1,\omega_2\colon
\mathbb H^+(B)\to\mathbb H^+(B)$ so that
\begin{enumerate}
\item $\Im\omega_j(b)\ge\Im b$ for all $b\in\mathbb H^+(B)$, $j\in\{1,2\}$;
\item $F_x(\omega_1(b))+b=F_y(\omega_2(b))+b=\omega_1(b)+\omega_2(b) \text{ for all } b\in\mathbb H^+(B);$
\item $G_x(\omega_1(b))=G_y(\omega_2(b))=G_{x+y}(b)\text{  for all } b\in\mathbb H^+(B).$
\end{enumerate}
Moreover, if $b\in\mathbb H^+(B)$, then $\omega_1(b)$ is the unique fixed point of the map
$$
f_b\colon\mathbb H^+(B)\to\mathbb H^+(B),\quad f_b(w)=h_y(h_x(w)+b)+b,
$$
and $\omega_1(b)=\lim_{n\to\infty}f_b^{\circ n}(w)$ for any $w\in\mathbb H^+(B)$. Same 
statements hold for $\omega_2$, with $f_b$ replaced by $w\mapsto h_x(h_y(w)+b)+b.$
\end{theorem}

Before starting the proof, let us give a brief outline of the reasoning and the main difficulties
we will encounter here compared to the scalar case proof from \cite{BB07}. There are two
separate parts of this proof: one shows that the iterations described in the Theorem
converge to analytic functions $\omega_1,\omega_2$ as claimed, and that
the equalities $F_x(\omega_1(b))=F_y(\omega_2(b))=\omega_1(b)+\omega_2(b)-b$ take
place. The other part shows that in fact $F_x(\omega_1(b))=F_{x+y}(b)$. Paradoxically,
in the operator-valued case, this second part is considerably more difficult and we will have to
dedicate to it several pages. The main idea will be to make a change of variable $o_1(b)=
(\omega_1(b^{-1}))^{-1}$ in order to be able to work on a neighbourhood of zero without
worrying about domains of inverses, use the $R$-transform to connect $o_1$ to 
$\mathcal C_{x+y}$, 
and finally prove
that there exists a nonempty open set in $\mathbb H^+(B)$ which is mapped by 
$o_1$ in $\mathbb H^+(B)$. This will allow us to come back to the functions $\omega$
and conclude that on this open subset $F_x\circ\omega_1=F_{x+y}$. We next extend
$\omega_1$ to all of $\mathbb H^+(B)$ as a limit of iterations of $f_b$ and use the 
uniqueness of analytic continuation to conclude.

\begin{proof} We shall begin by showing that functions $\omega_1,\omega_2$ satisfying
item (3) in our theorem exist and are well defined on some open set in $\mathbb H^+(B)$.
In order to do that, we shall use the change of variables advertised above. 
Recall from \eqref{omega-C} and \eqref{o1aroundzero} the definition $o_1(b)=\mathcal C_x^{\langle-1\rangle}
(\mathcal C_{x+y}(b)),$ and $o_2(b)=\mathcal C_y^{\langle-1\rangle}
(\mathcal C_{x+y}(b))$, for $b$ in some neighbourhood of zero. On this neighbourhood, we
can perform inversions of all functions $\mathcal C$, and so the definition \eqref{C-R} of the 
$R$-transform in terms of the function $\mathcal C$ together with \eqref{R-lineariz} allow us to write
\begin{equation}\label{12}
\left(b^{-1}\mathcal C^{\langle-1\rangle}_x(b)\right)^{-1}+\left(b^{-1}\mathcal C^{\langle-1\rangle}_y(b)\right)^{-1}=
{\bf1}+\left(b^{-1}\mathcal C^{\langle-1\rangle}_{x+y}(b)\right)^{-1},
\end{equation}
for any invertible $b$ of small norm (it is trivial that $b\mapsto b^{-1}\mathcal C^{\langle-1\rangle}(b)$ has an
analytic extension to a neighbourhood of zero which takes values close to one, hence is invertible
in $B$). 

Let us note that, from the definition and power series expansion \eqref{C} of $\mathcal C$, it
follows that both $b^{-1}\mathcal C_x(b)$ and $\mathcal C_x(b)b^{-1}$ have well-defined
analytic extensions to $B(0,\|x\|^{-1})$, regardless of the actual invertibility of $b$. For the
compositional inverse $\mathcal C_x^{\langle -1\rangle}$ of $\mathcal C_x$ the same holds true, as noted
in equation \eqref{C-R} (the reader can verify this also directly), on some (possibly smaller)
neighbourhood of the origin. Thus, we will take the liberty to denote the extension
of $b\mapsto b^{-1}\mathcal C_x(b)$ to $B(0,\|x\|^{-1})$ with the same notation
$b^{-1}\mathcal C_x(b)$, and similarly for $\mathcal C^{\langle-1\rangle}_x$: in the following, writing 
$b^{-1}\mathcal C_x(b)$ will imply no assumption on the invertibility of $b$.

By substituting $\mathcal C_{x+y}(b)$ for $b$ in \eqref{12}, we obtain
\begin{equation}\label{treize}
\left(\mathcal C_{x+y}(b)^{-1}\mathcal C^{\langle-1\rangle}_x(\mathcal C_{x+y}(b))\right)^{-1}+\left(
\mathcal C_{x+y}(b)^{-1}\mathcal C^{\langle-1\rangle}_y(\mathcal C_{x+y}(b))
\right)^{-1}={\bf 1}+b^{-1}\mathcal C_{x+y}(b),
\end{equation}
for any $b$ sufficiently small in $B$. 
Thus,
$$  
\mathcal C_{x+y}(b)^{-1}\mathcal C^{\langle-1\rangle}_y(\mathcal C_{x+y}(b))=\left[{\bf1}+b^{-1}\mathcal C_{x+y}(b)
-\left(\mathcal C_{x+y}(b)^{-1}\mathcal C^{\langle-1\rangle}_x(\mathcal C_{x+y}(b))\right)^{-1}\right]^{-1}
.$$
The expression under the square brackets is close to 1, hence invertible.
We recognize easily the 
functions $o$:
$$
\mathcal C_{x+y}(b)^{-1}o_2(b)=\left[{\bf1}+b^{-1}\mathcal C_{x+y}(b)
-\left(\mathcal C_{x+y}(b)^{-1}o_1(b)\right)^{-1}\right]^{-1}.
$$
Here an implicit statement is made regarding the analytic continuation to a neighbourhood of zero
of 
$\mathcal C_{x+y}(b)^{-1}o_2(b),\mathcal C_{x+y}(b)^{-1}o_1(b)$ and $(\mathcal C_{x+y}(b)^{-1}
o_1(b))^{-1}$: these statements follow from the analytic continuation  of
$o_1=\mathcal C_x^{\langle-1\rangle}\circ\mathcal C_{x+y}$ with $o_1(0)=0$, the 
analytic continuation of $w\mapsto\mathcal C_x(w)w^{-1}$, and the proximity to ${\bf1}$ of
$\mathcal C_x(w)w^{-1}$ when $\|w\|$ is small, on such a neighbourhood.

The power series expansion of $\mathcal C_x^{\langle-1\rangle}$ indicates \eqref{C-R} that for $\|b\|$ sufficiently small,
if $\mathcal C_{x+y}(b)$ is invertible in $B$, then so is $o_1(b)$. We conclude by analytic continuation
that
\begin{eqnarray}
o_2(b) & = & \mathcal C_{x}(o_1(b))\left[\left({\bf 1}+b^{-1}\mathcal C_{x}(o_1(b))
-\left(\mathcal C_x(o_1(b))^{-1}o_1(b)\right)^{-1}\right)\right]^{-1}\\
& = & \left[\left({\bf 1}+b^{-1}\mathcal C_{x}(o_1(b))
-\left(\mathcal C_x(o_1(b))^{-1}o_1(b)\right)^{-1}\right)\mathcal C_{x}(o_1(b))^{-1}
\right]^{-1}.\label{oo}
\end{eqnarray}
The first equality has the obvious meaning for all $b\in B(0,r)$, for some small enough $r>0$, 
while \eqref{oo} will be considered
on some fixed open subset of $B(0,r)$ on which $\mathcal C_{x+y}(b)$ is invertible (for example,
on some small neighbourhood of $\frac{ir}{M}{\bf 1}$ with $M>1$ large enough). In that case,
not only $\mathcal C_{x+y}(b)=\mathcal C_{x}(o_1(b))=\mathcal C_{y}(o_2(b))$ is invertible in $B$,
but so is $\mathcal C_{x}(o_1(b))^{-1}o_1(b)$ - being  close to one by equations
\eqref{C-R} and \eqref{omega-C} - and the whole parenthesis
$\left({\bf 1}+b^{-1}\mathcal C_{x}(o_1(b))
-\left(\mathcal C_x(o_1(b))^{-1}o_1(b)\right)^{-1}\right)$, again being  close to ${\bf 1}$
by the same argument.

Our next objective is to show that if $\Im b>0$ (for example, if $b$ belongs to a 
neighbourhood of
$\frac{ir}{M}{\bf 1}$ as described above), then $\Im o_1(b),\Im o_2(b)>0$. For this
purpose it will be convenient to express
$o_1$ as a fixed point in a formula involving the definition \eqref{C} of $\mathcal C$ as an 
expectation of a resolvent: by \eqref{oo}, $\mathcal C_x(o_1(b))=\mathcal C_y(o_2(b))$
is equivalent to 
$$
\mathcal C_x(o_1(b))=\mathcal C_y\left(\left[\left({\bf 1}+b^{-1}\mathcal C_{x}(o_1(b))
-\left(\mathcal C_x(o_1(b))^{-1}o_1(b)\right)^{-1}\right)
\mathcal C_{x}(o_1(b))^{-1}\right]^{-1}\right),
$$
again under the same hypotheses of invertibility of $\mathcal C_{x+y}(b)$. Rewrite this by 
using the definitions of $\mathcal C_x$ and $\mathcal C_y$ (see \eqref{C}):
\begin{eqnarray*}
\lefteqn{\mathbb E\left[({\bf 1}-o_1(b)x)^{-1}\right]o_1(b)}\\
& = & \mathbb E\left[\left({\bf 1}-\left[\left({\bf 1}+b^{-1}\mathcal C_{x}(o_1(b))
-\left(\mathcal C_x(o_1(b))^{-1}o_1(b)\right)^{-1}\right)\mathcal C_{x}(o_1(b))^{-1}\right]^{-1}
y\right)^{-1}\right]\\
& & \mbox{}\times\left[\left({\bf 1}+b^{-1}\mathcal C_{x}(o_1(b))
-\left(\mathcal C_x(o_1(b))^{-1}o_1(b)\right)^{-1}\right)\mathcal C_{x}(o_1(b))^{-1}\right]^{-1}\\
& = & \mathbb E\left[\left[\left({\bf 1}+b^{-1}\mathcal C_{x}(o_1(b))
-\left(\mathcal C_x(o_1(b))^{-1}o_1(b)\right)^{-1}\right)\mathcal C_{x}(o_1(b))^{-1}-y\right]^{-1}
\right].
\end{eqnarray*}
We assume $b\in B(0,r)$ has positive imaginary part, and hence is invertible. 
Since $\mathcal C$ is of the form $b({\bf 1}+v(b))$,
with $\lim_{b\to0}v(b)=0$, for $r$ sufficiently small, the invertibility of $b$ will imply the 
invertibility of all the three functions $\mathcal C$ involved, while the invertibility of $b\mapsto
\mathcal C_x(o_1(b))^{-1}o_1(b)$ is guaranteed by having chosen $r>0$ sufficiently small.  So we can rewrite the above as
\begin{multline}\label{eq:newinsert}
\mathbb E\left[({\bf 1}-o_1(b)x)^{-1}\right]o_1(b)\\
=\mathbb E\left[\left[\left({\bf 1}
-\left(\mathcal C_x(o_1(b))^{-1}o_1(b)\right)^{-1}\right)\mathcal C_{x}(o_1(b))^{-1}+b^{-1}-y\right]^{-1}
\right].
\end{multline}
By definition, 
$\mathcal C_x(w)=w\mathbb E\left[({\bf 1}-xw)^{-1}\right]$. For $\|w\|$ sufficiently small,
$\mathbb E\left[({\bf 1}-xw)^{-1}\right]$ is invertible, so the invertibility of $\mathcal C_x(w)$
in $B$ will imply the invertibility of $w$ whenever $\|w\|$ is small enough; indeed, the inverse
is trivially $$w^{-1}=\mathbb E\left[({\bf 1}-xw)^{-1}\right]\left(w\mathbb E\left[({\bf 1}-xw)^{-1}\right]\right)^{-1}.$$
Now since
$\mathcal C_x(w)^{-1}w=\left(w\mathbb E\left[({\bf 1}-xw)^{-1}\right]
\right)^{-1}w=\mathbb E\left[({\bf 1}-xw)^{-1}\right]^{-1}$, 
\begin{eqnarray*}
\left({\bf 1}-(\mathcal C_{x}(w)^{-1}w)^{-1}\right)\mathcal C_{x}(w)^{-1} 
&=& \left({\bf 1}-\mathbb E\left[({\bf 1}-xw)^{-1}\right]\right)\left(w\mathbb E\left[({\bf 1}-xw)^{-1}
\right]\right)^{-1}
\\
&= & -\mathbb E\left[({\bf 1}-xw)^{-1}x\right]w\left(\mathbb E\left[({\bf 1}-wx)^{-1}
\right]w\right)^{-1}\\
&=&\mathbb E\left[-x({\bf 1}-wx)^{-1}\right]\mathbb E\left[({\bf 1}-wx)^{-1}\right]^{-1}\\
&=&\mathbb E\left[x(wx-{\bf 1})^{-1}\right]\mathbb E\left[({\bf 1}-wx)^{-1}\right]^{-1},
\end{eqnarray*}
which has an analytic extension around zero. 
As for $\|b\|$ small 
enough, $\|o_1(b)\|$ is small,  invertibility of $\mathcal C_{x+y}(b)=
\mathcal C_x(o_1(b))=\mathcal C_y(o_2(b))$ in $B$ will imply the invertibility of $o_1(b),o_2(b)$.
Under the assumption that $\Im b>0$ and $\|b\|$ is small enough, so that invertibility of 
$\mathcal C_{x+y}(b)$ imply the invertibility of $o_1(b),o_2(b)$, we can
rewrite the above equality \ref{eq:newinsert} as
\begin{eqnarray*}
\lefteqn{\mathbb E\left[({\bf1}-o_1(b)x)^{-1}\right]o_1(b)=}\\
& & \mathbb E\left[\left(\mathbb E\left[x\left(o_1(b)x-{\bf1}\right)^{-1}\right]\mathbb E\left[\left({\bf1}-o_1(b)x\right)^{-1}\right]^{-1}+b^{-1}-y\right)^{-1}
\right].
\end{eqnarray*}
Dividing by the (invertible) element $\mathbb E\left[({\bf1}-o_1(b)x)^{-1}\right]\in B$
gives 
\begin{equation}\label{10}
o_1(b)=\mathbb E\left[\left(\mathbb E\left[x(o_1(b)x-{\bf1})^{-1}\right]+(b^{-1}-y)
\mathbb E\left[({\bf1}-o_1(b)x)^{-1}\right]\right)^{-1}
\right].
\end{equation}
We will argue now that for elements 
$b\in\mathbb H^+(B)$ with the property that $-\Im(b^{-1})$ is very large, while $\Re(b^{-1})$
stays bounded, $o_1(b)\in\mathbb H^+(B)$. This will follow very easily from equation
\eqref{10} as soon as we prove the following two 

\noindent{\bf Claims:}
\begin{enumerate}
\item For a given $k\in(0,+\infty)$ and $v\in\mathbb H^+(\mathcal M)$ with $\Im v\ge k{\bf1}$, 
there exists  $\varepsilon=\frac{k}{2\|v\|}>0$ so that
$\|c-{\bf1}\|<\varepsilon\implies \Im(vc)\ge\frac{k}{2}{\bf 1}$.
\item For any $\varepsilon>0$ and $x=x^*\in\mathcal M$ there exists a $\delta>0$ so that for
any $w\in\mathcal M$, $\|w\|<\delta\implies
-\varepsilon{\bf1}<\Im\mathbb E\left[x({\bf1}-wx)^{-1}\right]<\varepsilon{\bf1}.$
\end{enumerate}
We prove claim (1) first.  We have
\begin{eqnarray*}
\Im(vc)&=&\frac{1}{2i}(vc-c^*v^*)\\
&=&\frac{1}{2i}(v(c-{\bf1})+v-(c^*-{\bf1})v^*-v^*)\\
&=&\frac{1}{2i}\bigl(v(c-{\bf1})-(c-{\bf1})^*v^*\bigr)+\Im v.
\end{eqnarray*}
Since for any selfadjoint $a$ we have $-\|a\|{\bf1}\leq a\leq\|a\|{\bf1}$, we get
$$-\frac12\|v(c-{\bf1})-(c-{\bf1})^*v^*\|{\bf1}\leq\frac{1}{2i}(v(c-{\bf1})-(c-{\bf1})^*v^*)\leq\frac12\|v(c-{\bf1})-(c-{\bf1})^*v^*\|{\bf1}.$$
It follows that $-\|v\|\|c-{\bf1}\|\le\frac{1}{2i}(v(c-{\bf1})-(c-{\bf1})^*v^*)\le\|v\|\|c-{\bf1}\|$, i.e.
$$
-\varepsilon\|v\|{\bf1}\le\frac{1}{2i}(v(c-{\bf1})-(c-{\bf1})^*v^*)\le\varepsilon\|v\|{\bf1},
$$
and so 
$$
\Im v-\varepsilon\|v\|{\bf1}\le\frac{1}{2i}(v(c-{\bf1})-(c-{\bf1})^*v^*)+\Im v=\Im (vc)
\le\varepsilon\|v\|{\bf1}+\Im v
$$
Thus picking the positive $\varepsilon=\frac{k}{2\|v\|}$ will 
guarantee that $\Im v-\frac{k}{2}{\bf 1}\leq\Im(vc)\leq\Im v+\frac{k}{2}{\bf 1}$. 
This proves the first claim.

Claim (2) is quite trivial: observe that
$\Im(x({\bf1}-wx)^{-1})=(xw^*-{\bf1})^{-1}x(\Im w )x(wx-{\bf1})^{-1}$, so making $\|w\|$ 
sufficiently small will provide the desired result.

 We apply the first claim to $v=y-b^{-1}$ and $c=\mathbb E\left[(1-o_1(b)x)^{-1}\right]$:
this means we require that $\|c-{\bf1}\|<M/2\|b^{-1}-y\|$, while at the same time $-\Im (b^{-1})
>M{\bf1}>{\bf1}$. For this to happen, it is enough that $\frac{\|o_1(b)x\|}{1-\|o_1(b)x\|}<\frac{M}{2\|b^{-1}-y\|}$,
or, equivalently, $\|o_1(b)x\|<\frac{M}{M+2\|b^{-1}-y\|}$. From the similar property of $\mathcal C$,
we know that for any $\varepsilon>0$ there exists a $\delta=\delta_{\varepsilon,x,y}>0$ so that 
$\|b\|<\delta\implies\|o_1(b)\|<\|b\|(1+\varepsilon)$. 
Now, for $\|o_1(b)x\|<\frac{M}{M+2\|b^{-1}-y\|}$ to happen, it is enough
that $\|o_1(b)\|\|x\|<\frac{M}{M+2\|\Im(b^{-1})\|+2\|\Re(b^{-1})\|+2\|y\|}$. We will restrict 
$\|\Re(b^{-1})\|\in(0,1)$ and $\|\Im(b^{-1})\|\in(M,2M)$ (these two inequalities together with the
requirement $-\Im(b^{-1})>M{\bf1}$ determine an open nonempty set). This will make our condition 
$\|o_1(b)\|\|x\|<\frac{M}{M+2\|\Im(b^{-1})\|+2\|\Re(b^{-1})\|+2\|y\|}$ to be implied by
$\|o_1(b)\|\|x\|<\frac{M}{5M+2+2\|y\|}$. The conditions on $\Re,\Im (b^{-1})$ imply
that $\|b\|\le\frac{1}{M}$. As $M\to\infty$, this will send $\|b\|$ to zero, and so for any chosen 
$\varepsilon>0$, if $M$ is sufficiently large, $\|o_1(b)\|<\frac{3(1+\varepsilon)}{2M}<
\frac{M}{\|x\|(5M+2+2\|y\|)}$.
This shows that the first claim can be applied to $v=y-b^{-1}$ and $c=\mathbb E\left[({\bf1}-
o_1(b)x)^{-1}\right]$ to conclude that for $b$ in an open set having zero in its closure
we have $\Im((y-b^{-1})\mathbb E\left[({\bf1}-o_1(b)x)^{-1}\right])\ge{\bf1}$.

By making the $M$ above sufficiently large, we finally apply the second claim to 
$w=o_1(b)$ and our given $x$ to obtain the inequalities
$-\frac12{\bf1}<\Im\mathbb E\left[x({\bf1}-o_1(b)x)^{-1}\right]<\frac12{\bf1}$.
We conclude that for sufficiently large $M\in(0,+\infty)$ and $b$ so that 
$\|\Re(b^{-1})\|\in(0,1)$, $\|\Im(b^{-1})\|\in(M,2M)$ and $-\Im(b^{-1})>M{\bf1}$,
$$
\Im\left(\mathbb E\left[x({\bf1}-o_1(b)x)^{-1}\right]+(y-b^{-1})
\mathbb E\left[({\bf1}-o_1(b)x)^{-1}\right]\right)>0
$$
By the positivity of $\mathbb E$, the
fact that $w\mapsto w^{-1}$ maps $\mathbb H^+(B)$ bijectively onto $\mathbb H^-(B)$,
and equation \eqref{10}, we conclude that for such
$b\in\mathbb H^+(B)$ we have $\Im o_1(b)>0$. Relation \eqref{oo} can then be re-written
as
\begin{equation}\label{odoi}
o_2(b)=\left[\mathcal C_x(o_1(b))^{-1}-o_1(b)^{-1}+b^{-1}\right]^{-1}.
\end{equation}
Recalling that $\mathcal C_x(o_1(b))^{-1}=F_x(o_1(b)^{-1})$ and $\Im F_x(w)\leq\Im w$ when
$w\in\mathbb H^-(B)$, it follows that $\Im(\mathcal C_x(o_1(b))^{-1}-o_1(b)^{-1}+b^{-1})<0,$
so $\Im o_2(b)>0$.

Since $o_j(b)=(\omega_j(b^{-1}))^{-1}$, $j\in\{1,2\}$, we found elements
$b\in\mathbb H^+(B)$ so that $\omega_j(b)\in\mathbb H^+(B)$, and so we can write
$G_x(\omega_1(b))=G_{x+y}(b)=G_y(\omega_2(b))$.
This proves the existence of the functions omega on a common set in the operator
upper half-plane of $B$.

We go now to the second part of the proof, namely that $\lim_{n\to\infty}f_b^{\circ n}(w)$ exists for all 
$b,w\in\mathbb H^+(B),$ depends only on $b$ (not on $w$), and is the attracting fixed point of
$f_b$. To do that, let us start by proving that for any $\varepsilon>0$, $h_y(\mathbb H^+(B)+
i\varepsilon{\bf1})$ is norm bounded, with bound depending only on $\varepsilon$ and $\|y\|$. 
For the reader's convenience, we shall state and prove this fact in a separate lemma.

\begin{lemma}\label{2.3}
If $(\mathcal M,\mathbb E,B)$ is a $C^*$-operator-valued non-commutative probability
space, $y=y^*\in\mathcal M$ is a $B$-valued selfadjoint random variable and
$$
h_y\colon\mathbb H^+(B)\to\overline{\mathbb H^+(B)},\quad
h_y(w)=\mathbb E\left[\left(w-y\right)^{-1}\right]^{-1}-w,
$$
then for all $w\in \mathbb H^+(B)+i{\varepsilon}{\bf1}$, we have 
$$
\|h_y(y)\|\leq4\|y\|(1+2\varepsilon^{-1}\|y\|).
$$
\end{lemma}
\begin{proof}
We shall
use the same trick as in \cite{BSTV}, namely for a fixed $w=u+iv\in\mathbb H^+(B)$ so that 
$v>\varepsilon1$ and $\varphi\colon B\to\mathbb C$ positive linear functional, we define the analytic functions
$$
f_{\varphi,w}\colon\mathbb C^+\to\mathbb C^+\cup\mathbb R,\quad
f_{\varphi,w}(z)=\varphi(h_y(u+zv)).
$$
A straightforward computation (already performed in \cite[Lemma 2.3]{BSTV}) shows that
$$
\lim_{z\to\infty}h_y(u+zv)=-\mathbb E[y],\quad\lim_{z\to\infty}z(h_y(u+zv)+\mathbb E[y])=
\mathbb E[y]v^{-1}\mathbb E[y]-\mathbb E[yv^{-1}y].
$$
The continuity of positive linear functionals on $C^*$-algebras allows us to conclude that
$$
\lim_{z\to\infty}f_{\varphi,w}(z)=-\varphi(\mathbb E[y]),
\quad\lim_{z\to\infty}z(f_{\varphi,w}(z)+\varphi(\mathbb E[y]))=\varphi(
\mathbb E[y]v^{-1}\mathbb E[y]-\mathbb E[yv^{-1}y]).
$$
Thus, by the Nevanlinna representation \cite[Chapter III]{akhieser}, there exists a
positive compactly supported measure $\rho$ on the real line with
$\rho(\mathbb R)=\varphi(\mathbb E[yv^{-1}y]-\mathbb E[y]v^{-1}\mathbb E[y])$ so that
$$
f_{\varphi,w}(z)=-\varphi(\mathbb E[y])+\int_\mathbb R\frac{1}{t-z}\,d\rho(t),\quad \Im z>0.
$$
This gives a bound for $f_{\varphi,w}(z)$ as follows:
\begin{eqnarray*}
|f_{\varphi,w}(z)|&\leq&|-\varphi(\mathbb E[y])|+\frac{1}{\Im z}\rho(\mathbb R)\\
&\le&\|\varphi\|\|y\|+\frac{1}{\Im z}\|\varphi\|\|\mathbb E[yv^{-1}y]-
\mathbb E[y]v^{-1}\mathbb E[y]\|.
\end{eqnarray*}
Since $\varphi(h_y(w))=f_{\varphi,w}(i)$, it follows that
$$
|\varphi(h_y(w))|\leq\|\varphi\|\left(\|y\|+\|\mathbb E[yv^{-1}y]-\mathbb E[y]v^{-1}\mathbb E[y]\|\right)<\|\varphi\|\|y\|\left(1+2\varepsilon^{-1}\|y\|\right),
$$
for any linear positive $\varphi$ and any $w\in\mathbb H^+(B)$ with $\Im w>\varepsilon1$.
Since, by the Jordan decomposition, any continuous linear functional on $B$ is a sum of four 
positive linear functionals, it follows that
$$
\sup_{\|\varphi\|<1}|\varphi(h_y(w))|\le4\|y\|\left(1+2\varepsilon^{-1}\|y\|\right),
$$ 
for all $w\in\mathbb H^+(B)+i\varepsilon{\bf1}$.
Since $B$ is in particular a Banach space, we know that it embeds isometrically in its bidual, so
$\|b\|=\sup_{\|\varphi\|<1}|\varphi(b)|$ for any $b\in B$, where $\varphi$ runs through the
dual of $B$ and the norm is the norm on the dual of $B$. Thus,
$$
\|h_y(w)\|\leq4\|y\|\left(1+2\varepsilon^{-1}\|y\|\right),\quad\text{for } w\in\mathbb H^+(B),\Im w>
\varepsilon{\bf 1},
$$
which proves our lemma.
\end{proof}

We shall use this lemma to prove that for any fixed $b\in\mathbb H^+(B)$, $\Im b\ge
\varepsilon{\bf 1}$, there is an $m>0$ depending on $b$, $x$ and $y$ so that 
\begin{enumerate}
\item[(a)] $f_b$ maps $B(0,2m)\cap\left(\mathbb H^+(B)+i\frac{\varepsilon}{2}{\bf1}\right):=
\left\{w\in\mathbb H^+(B)+i\frac{\varepsilon}{2}{\bf1}\colon\|w\|<2m\right\}$
into itself, and
\item[(b)] the inclusion is strict, in the (stronger) sense that
\begin{eqnarray}
\lefteqn{
\inf\left\{\|u-v\|\colon u\in f_b\left(B(0,2m)\cap\left(\mathbb H^+(B)+i\frac{\varepsilon}{2}{\bf1}\right)\right),
\right.}\\
& & \left.\frac{}{}\frac{}{}\frac{}{}\frac{}{}\frac{}{}
\frac{}{}\frac{}{}\frac{}{}\frac{}{}\frac{}{}\frac{}{}\frac{}{}\frac{}{}\frac{}{}\frac{}{}\frac{}{}\frac{}{}\frac{}{}
\frac{}{}\frac{}{}\frac{}{}\frac{}{}\frac{}{}\frac{}{}\frac{}{}\frac{}{}\frac{}{}\frac{}{}\frac{}{}\frac{}{}
\frac{}{}\frac{}{}\frac{}{}\frac{}{}\frac{}{}\frac{}{}\frac{}{}\frac{}{}\frac{}{}\frac{}{}\frac{}{}\frac{}{}
v\in B\setminus\left[B(0,2m)\cap\left(\mathbb H^+(B)+i\frac{\varepsilon}{2}
{\bf1}\right)\right]\right\}>0.\nonumber
\end{eqnarray}
\end{enumerate}
Fix a $b\in\mathbb H^+(B)$, and choose $\varepsilon>0$ so that $\Im b\ge\varepsilon{\bf 1}$.
Then, by the above lemma,
$$
\sup_{w\in\mathbb H^+(B)}\|f_b(w)\|\leq\|b\|+\sup_{w\in\mathbb H^+(B)+i\frac{\varepsilon}{2}
1}\|h_y(w)\|\le \|b\|+4\|y\|\left(1+4\varepsilon^{-1}\|y\|\right)<\infty.
$$
We shall take 
\begin{equation}\label{m}
m:=\|b\|+4\|y\|\left(1+4\varepsilon^{-1}\|y\|\right).
\end{equation}
On the other hand, as mentioned above, $\Im h_y(w)\ge0$ for all $w\in\mathbb H^+(B)$, so that
$\Im f_b(w)\ge\Im b\ge\varepsilon{\bf1}>\frac{\varepsilon}{2}{\bf1}$. This guarantees that $
\left\{w\in\mathbb H^+(B)+i\frac{\varepsilon}{2}1\colon\|w\|<2m\right\}$ is mapped
by $f_b$ into itself, proving item (a) above.
We shall give an explicit lower bound in order to show item (b):
if a  $w_0\not\in \left\{w\in\mathbb H^+(B)+i\frac{\varepsilon}{2}1\colon\|w\|<2m\right\}$,
then either $\|w_0\|\ge2m$, and then $\|w_0-f_b(w)\|\ge\|w_0\|-\|f_b(w)\|\ge m$ for all
$w\in\mathbb H^+(B)+i\frac{\varepsilon}{2}1$, or $\Im w_0\not>\frac{\varepsilon}{2}1,$ and
then\footnote{If $a\ge\varepsilon1$ and $b=b^*\not>\frac{\varepsilon}{2}1$ then (using the 
GNS representation of $B$ on a Hilbert space) there exists a vector $\xi$ of norm one so that
$(b\xi|\xi)\le\varepsilon(\xi|\xi)/2=\varepsilon/2$. But $(a\xi|\xi)\ge\varepsilon$, so
$((a-b)\xi|\xi)\ge\varepsilon/2$, which implies $\|a-b\|\ge\varepsilon/2$.
}
$$
\|f_b(w)-w_0\|\ge\|\Im f_b(w)-\Im w_0\|\ge\varepsilon/2.
$$
We conclude that the distance between $f_b\left(\left\{w\in\mathbb H^+(B)+i\frac{\varepsilon}{2}1\colon\|w\|<2m\right\}\right)$ and $B\setminus\left(
\left\{w\in\mathbb H^+(B)+i\frac{\varepsilon}{2}1\colon\|w\|<2m\right\}\right)$ is no less than
$\min\{m,\varepsilon/2\}$, a strictly positive number.

This allows us to apply the Earle-Hamilton Theorem \cite[Theorem 11.1]{Din}
to conclude that $f_b$ has a unique
{\em attracting} fixed point $w_b$ in the set
$\{w\in\mathbb H^+(B)+i\varepsilon1\colon\|w\|<m\}$, with $m=\|b\|+4\|y\|
\left(1+4\varepsilon^{-1}\|y\|\right)$, constant provided by Lemma \ref{2.3}. 
According to this theorem, the map $f_b$ is a strict 
contraction in the Carath\'eodory metric (not necessarily in the norm metric); however, since
the two metrics are comparable on sets strictly inside the given domain (i.e. at strictly positive
distance from the complement of the domain), it follows that 
$\lim_{n\to\infty}f_b^{\circ n}(w)=w_b$ happens both in the Carath\'eodory and the norm
metric whenever $w$ is chosen strictly inside the domain of $f_b$. Thus, as in the
proof of \cite[Theorem 2.2]{BSTV}, it follows that the map $b\mapsto w_b$ is
G\^{a}teaux analytic, and, by its - global, not only local - boundedness on 
bounded domains in $\mathbb H^+(B)+i\frac{\varepsilon}{2}{\bf1}$ for any $\varepsilon>0$,
it is Fr\'echet analytic - see \cite{Dineen}. Since this point $w_b$ is in 
$\overline{\mathbb H^+(B)+b}$
whenever $\Im b\ge\varepsilon{\bf 1}$, item (1) of our theorem is proved.

To conclude our proof, we shall show next that for elements $b$ as in the first part of our proof,
$w_b=(o_1(b^{-1})^{-1}=\omega_1(b)$. The uniqueness of analytic continuation will allow us to
conclude. Denote by $w'_b$ the fixed point of the map $w\mapsto h_x(h_y(w)+b)+b$. 
If $w_b=h_y(h_x(w_b)+b)+b$ and $w'_b=h_x(h_y(w'_b)+b)+b$,
then $h_x(w_b)+b=h_x(h_y(h_x(w_b)+b)+b)+b$, and, by the uniqueness of the fixed point in the
upper half-plane, $h_x(w_b)+b=w'_b$. Similarly, $h_y(w_b')+b=w_b$.
Adding $w_b$ to the first and $w'_b$ to the second equation gives us
$w_b+w'_b=b+F_x(w_b)=b+F_y(w'_b)$. Thus, the fixed points in question satisfy item (2)
of our theorem. On the other hand, by equation \eqref{odoi}, it follows that
$o_2(b)^{-1}=\mathcal C_x(o_1(b))^{-1}+b^{-1}-o_1(b)^{-1}=h_x(o_1(b)^{-1})+b^{-1}$ 
(recall that $\mathcal C_x(w)^{-1}=F_x(w^{-1})
$) whenever $b\in\mathbb H^-(B)$ of
small enough norm is so that $\Im o_1(b)<0$ (which, as seen in \eqref{odoi}, implies that 
$\Im o_2(b)<0$). Writing now $\mathcal C_x(o_1(b))=\mathcal C_y(o_2(b))$ gives us
$$
F_x(o_1(b)^{-1})=\mathcal C_x(o_1(b))^{-1}=\mathcal C_y((h_x(o_1(b)^{-1})+b^{-1})^{-1})^{-1}=
F_y(h_x(o_1(b)^{-1})+b^{-1}).
$$
Adding $o_1(b)^{-1}$ to both sides and recalling the definition of $h$ finally gives us
$$
o_1(b)^{-1}=h_y(h_x(o_1(b)^{-1})+b^{-1})+b^{-1},
$$
for $b\in\mathbb H^-(B)$ so that $\Im o_1(b)<0$ - the existence of an open set on which this
happens having been proved in the first part of our proof. We perform the change of variable 
$b\mapsto b^{-1}$ to conclude that the correspondence $b\mapsto o_1(b^{-1})^{-1}$ is a
fixed point in $\mathbb H^+(B)$ for $f_b$, so, by uniqueness of the fixed point of $f_b$, must
coincide with $w_b$. Since on this small open set, $b\mapsto o_1(b^{-1})^{-1}=\omega_1(b)$
satisfies also item (3), we conclude by analytic continuation that (3) holds for all $b\in
\mathbb H^+(B)$. This concludes the proof.
\end{proof}

The methods invoked in the proof of our main theorem will also be of use
for further investigations on properties of operator-valued distributions and their convolutions. As one such instance
let us to record here, for future use,
an observation about how one can recognize ``atoms'' of operator-valued distributions from properties of operator-valued resolvents.

\begin{proposition}
If $(\mathcal M,\mathbb E,B)$ is a $W^*$-operator-valued noncommutative probability space, $B$ is finite
dimensional, $y=y^*\in\mathcal M$ and $\mathbb E$ is faithful, then either
$\Im\mathbb E\left[(w-y)^{-1}\right]^{-1}>\Im w$ for all $w\in\mathbb H^+(B)$ or there
exists a projection $p\in B\setminus\{0\}$ so that $py=p\mathbb E[y]$. Moreover, then
$\ker\Im\left(\mathbb E\left[(w-y)^{-1}\right]^{-1}- w\right)$ is independent of $w$ and equals
the largest $p\in B$ which satisfies $py=p\mathbb E[y]$ and $yp=\mathbb E[y]p$.
\end{proposition}

As the reader will surely note, when $B=\mathbb C$ in
the above proposition, the relation $py=p\mathbb E[y]$
reduces to the equality between $y$ and its expectation,
making its distribution a point mass. This justifies our
reference above to ``atoms". When the dimension of $B$ is
infinite, the situation obviously becomes intractable,
as it will be evident from our proof below.

\begin{proof}
We know from \cite{BPV2010} that $\Im\mathbb E\left[(w-y)^{-1}\right]^{-1}\ge\Im w$ for all
$w\in\mathbb H^+(B)$.
We shall check
under what conditions the inequality fails to be strict. Since $B$ is finite dimensional,
non-strict inequality will imply the existence of a vector $\xi\in L^2(B)$ of norm one and a
$w\in\mathbb H^+(B)$ so that
$$
\Im\left(\mathbb E\left[(w-y)^{-1}\right]^{-1}\xi|\xi\right)=\Im(w\xi|\xi),
$$
and, in particular, the existence of a projection $p\in B\setminus\{0\}$ satisfying the equality
$p\left(\Im\mathbb E\left[(w-y)^{-1}\right]^{-1}\right)p=p(\Im w)p.$
We use now the same idea as in the proof of Lemma \ref{2.3}: consider the map
$$
z\mapsto\varphi\left(\mathbb E\left[(z\Im w+(\Re w-y))^{-1}\right]^{-1}\right),\quad
z\in\mathbb C^+,
$$
for any state $\varphi$ on $B$. This is again an analytic map from $\mathbb C^+$ to $\mathbb
C^+\cup\mathbb R$. Since
$\lim_{z\to\infty}\frac{\mathbb E\left[(z\Im w+(\Re w-y))^{-1}\right]^{-1}}{z}=
\mathbb E\left[(\Im w)^{-1}\right]^{-1}=\Im w,$
$$
\lim_{z\to\infty}\mathbb E\left[(z\Im w+(\Re w-y))^{-1}\right]^{-1}-z\Im w=
\Re w-\mathbb E\left[y\right],
$$
and
\begin{eqnarray*}
\lefteqn{\lim_{z\to\infty}z\left(\Re w-\mathbb E\left[y\right]+z\Im w-\mathbb E\left[(z\Im
w+(\Re w-y))^{-1}\right]^{-1}\right)}\\
& = &\lim_{z\to\infty}z\left\{(\Re w-\mathbb E\left[y\right]+z\Im w)\mathbb E\left[(z\Im w+(\Re
w-y))^{-1}\right]-{\bf 1}\right\}\\
& & \mbox{}\times\mathbb E\left[(z\Im w+(\Re w-y))^{-1}\right]^{-1}\\
& = & \lim_{z\to\infty}z^2\mathbb E\left[(\Re w-\mathbb E\left[y\right]+z\Im w)
(z\Im w+(\Re w-y))^{-1}-{\bf1}\right]\\
& & \mbox{}\times\lim_{z\to\infty}\frac{\mathbb E\left[(z\Im w+(\Re w-y))^{-1}\right]^{-1}}{z}\\
&=& \lim_{z\to\infty}z^2\mathbb E\left[(\Re w-\mathbb E\left[y\right]+z\Im w)
\left\{(z\Im w)^{-1}+(z\Im w)^{-1}(y-\Re w)(z\Im w)^{-1}\right.\right.\\
& & \mbox{}+\left.\left.(z\Im w+\Re w-y)^{-1}(\Re w-y)(z\Im w)^{-1}(\Re w-y)(z\Im w)^{-1}
\right\}-1\right]\Im w\\
& = & \lim_{z\to\infty}z^2\mathbb E\left[(y-\Re w)(z\Im w)^{-1}-(\mathbb E[y]-\Re w)
(z\Im w)^{-1}\right]\Im w\\
& & \mbox{}+\lim_{z\to\infty}z^2\mathbb E\left[{\bf1}-(\mathbb E[y]-\Re w)
(z\Im w)^{-1}(y-\Re w)(z\Im w)^{-1}-{\bf1}\right]\Im w\\
& & \mbox{}+\lim_{z\to\infty}z^2\mathbb E\left[(\Re w-\mathbb E\left[y\right]+z\Im w)
(z\Im w+\Re w-y)^{-1}(\Re w-y)(z\Im w)^{-1}\right.\\
& & \mbox{}\times\left.(\Re w-y)(z\Im w)^{-1}\right]\Im w\\
& = & \mathbb E\left[(\Re w-y)(\Im w)^{-1}(\Re w-y)\right]-(\mathbb E[y]-\Re w)(\Im w)^{-1}
(\mathbb E[y]-\Re w)\\
& = & \mathbb E\left[y(\Im w)^{-1}y\right]-\mathbb E[y](\Im w)^{-1}\mathbb E[y]\\
&=&\mathbb E\left[(y-\mathbb E[y])(\Im w)^{-1}(y-\mathbb E[y])\right].
\end{eqnarray*}
(We have used in the third equality above the easily verified identity
$(b-a)^{-1}=b^{-1}+b^{-1}ab^{-1}+(b-a)^{-1}ab^{-1}ab^{-1}$ with $a=y-\Re w$ and $b=
z\Im w$.) It follows that
$$
\varphi\left(\mathbb E\left[(z\Im w+(\Re w-y))^{-1}\right]^{-1}\right)=
\varphi(\Re w-\mathbb E[y]+z\Im w)+\int_\mathbb R\frac{1}{t-z}\,d\rho_{\varphi,w}(t),
$$
where $\rho_{\varphi,w}$ has compact support and $\rho_{\varphi,w}(\mathbb R)=
\varphi(\mathbb E\left[(y-\mathbb E[y])(\Im w)^{-1}(y-\mathbb E[y])\right])$.
So, by taking $z=i$ in the above, we have translated the question when we have equality
$\Im\varphi(\mathbb E\left[(w-y)^{-1}\right]^{-1})=\Im \varphi(w)$ for a given $\varphi$ of the
special form $\varphi(\cdot)=(\cdot\xi|\xi)$ (i.e., for an extremal state) into the question: when can the equality
$\varphi(\mathbb E[(y-\mathbb E[y])(\Im w)^{-1}(y-\mathbb E[y])])=0$ take place. If it does
take place at all, it trivially must take place on at least one, possibly more, extremal
states. That and the faithfulness of $\mathbb E$ imply the existence of a projection $p\in B$ so
that $p(y-\mathbb E[y])(\Im w)^{-1}(y-\mathbb E[y])p=0$, and in particular,
$p(y-\mathbb E[y])(\Im w)^{-1/2}=(\Im w)^{-1/2}(y-\mathbb E[y])p=0$. The invertibility
of $(\Im w)^{-1/2}$ allows us to simplify it, and so
$$ p(y-\mathbb E[y])=(y-\mathbb E[y])p=0,
$$
as claimed in our proposition.

Observe that if a projection $p\in B$ as in our lemma exists,  for a $w\in
\mathbb H^+(B)$ we have
$p\mathbb E\left[(y-\mathbb E[y])(\Im w)^{-1}(y-\mathbb E[y])\right]
=\mathbb E\left[(py-p\mathbb E[y])(\Im w)^{-1}(y-\mathbb E[y])\right]=
\mathbb E\left[(y-\mathbb E[y])(\Im w)^{-1}(y-\mathbb E[y])\right]p=0$.
In particular, $p\Im\mathbb E[(w-y)^{-1}]^{-1}=p\Im w=\Im wp=\Im\mathbb E[(w-y)^{-1}]^{-1}
p,$ which shows us that $p$ is a projection majorized by the kernel of
$\Im\left(\mathbb E[(w-y)^{-1}]^{-1}-w\right)$, and conversely, any projection
majorized by $\ker\Im\left(\mathbb E[(w-y)^{-1}]^{-1}-w\right)$ satisfies $py=p\mathbb E[y]$.
Of course, we shall next take the largest projection $p\in B$ so that $py=p\mathbb E[y]$.
The above shows us that $\ker\Im\left(\mathbb E[(w-y)^{-1}]^{-1}-w\right)$ must be
in fact independent of $w$, and hence equal to this $p$! This completes the proof.
\end{proof}

We observe that the above proposition indicates that whenever
$w\mapsto\mathbb E[(w-y)^{-1}]^{-1}-w$ does not map the upper half-plane inside
itself, by taking $w=iy\cdot{\bf1}$, we can immediately identify the projection $p$
out of the corresponding limit at infinity as either $\ker(y-\mathbb E[y])^2$
or $\ker(\mathbb E[y^2]-\mathbb E[y]^2)$.

\section{Anderson's selfadjoint linearization trick}

Let $\A$ be a complex and unital $\ast$-algebra and let selfadjoint elements $x_1,\dots,x_n\in\A$ be given. Then, for any non-commutative polynomial $p\in \C\langle X_1,\dots,X_n\rangle$, we get an operator $P=p(x_1,\dots,x_n)\in\A$ by evaluating $p$ at $(x_1,\dots,x_n)$.

In this situation, knowing a ``linearization trick'' means to have a procedure that leads finally to an operator
$$L_P = b_0 \otimes 1 + b_1 \otimes x_1 + \dots + b_n \otimes x_n \in M_N(\C)\otimes\A$$
for some matrices $b_0,\dots,b_n\in M_N(\C)$ of dimension $N$, such that $z-P$ is invertible in $\A$ if and only if $\Lambda(z)-L_P$ is invertible in $ M_N(\C)\otimes\A$. Hereby, we put
\begin{equation}\label{Lambda}
\Lambda(z) = \begin{bmatrix} z & 0 & \dots & 0\\ 0 & 0 & \dots & 0\\ \vdots & \vdots & \ddots & \vdots\\ 0 & 0 & \hdots & 0 \end{bmatrix} \qquad\text{for all $z\in\C$}.
\end{equation}
As we will see in the following, the linearization in terms of the dimension $N\in\N$ and the matrices $b_0,\dots,b_n\in M_N(\C)$ usually depends only on the given polynomial $p\in\C\langle X_1,\dots,X_n\rangle$ and not on the special choice of elements $x_1,\dots,x_n\in\A$.

The first famous linearization trick goes back to Haagerup and Thorbj\o rnsen \cite{HaagerupThorbjornsen1}, \cite{HaagerupThorbjornsen2} and turned out to be a powerful tool in many different respects. However, there was the disadvantage that, even if we start from an selfadjoint polynomial operator $P$, in general, we will not end up with a linearization $L_P$, which is selfadjoint as well.
Then, in \cite{Anderson1} and \cite{Anderson2}, Anderson presented a new version of this linearization procedure, which was able to preserve selfadjointness.

For readers convenience and since this will be one of our main tools, we recall here the main facts about Anderson's selfadjoint linearization trick, streamlined to our needs.

\begin{definition}\label{linearization}
Let $p\in \C\langle X_1,\dots,X_n\rangle$ be given. A matrix
$$L_p := \begin{bmatrix} 0 & u\\ v & Q\end{bmatrix}\in M_N(\C)\otimes\C\langle X_1,\dots,X_n\rangle,$$
where
\begin{itemize}
 \item $N\in\N$ is an integer,
 \item $Q\in M_{N-1}(\C)\otimes\C\langle X_1,\dots,X_n\rangle$ is invertible
 \item and $u$ is a row vector and $v$ is a column vector, both of size $N-1$ with entries in $\C\langle X_1,\dots,X_n\rangle$,
\end{itemize}
is called \textbf{a linearization of $p$}, if the following conditions are satisfied:
\begin{itemize}
 \item[(i)] There are matrices $b_0,\dots,b_n\in M_N(\C)$, such that $$L_p = b_0\otimes 1 + b_1 \otimes X_1 + \dots + b_n \otimes X_n,$$ i.e. the polynomial entries in $Q$, $u$ and $v$ all have degree $\leq1$.
 \item[(ii)] It holds true that $$p = - uQ^{-1}v.$$
\end{itemize}
\end{definition}

From condition (i) it seems natural to call such an operator $L_p$ a linearization of $p$. But at the first glance, there seems to be no reason for assuming additionally a certain block structure of $L_p$ and moreover condition (ii). The following well-known result about Schur complements will make this clear.

\begin{proposition}\label{Schur}
Let $\A$ be a complex and unital algebra. Let matrices $a\in M_k(\A)$, $b\in M_{k\times l}(\A)$, $c\in M_{l\times k}(\A)$ and $d\in M_l(\A)$ be given and assume that $d$ is invertible in $M_l(\A)$. Then the following statements are equivalent:
\begin{itemize}
 \item[(i)] The matrix $\begin{bmatrix} a & b\\ c & d\end{bmatrix}$ is invertible in $M_{k+l}(\A)$.
 \item[(ii)] The \textbf{Schur complement} $a-bd^{-1}c$ is invertible in $M_k(\A)$.
\end{itemize}
If the equivalent conditions (i) and (ii) are satisfied, we have the relation
\begin{equation}\label{Schur-formula}
\begin{bmatrix} a & b\\ c & d\end{bmatrix}^{-1} = \begin{bmatrix} 0 & 0\\ 0 & d^{-1}\end{bmatrix} + \begin{bmatrix} 1\\ -d^{-1}c\end{bmatrix} (a-bd^{-1}c)^{-1} \begin{bmatrix} 1 & -bd^{-1}\end{bmatrix}.
\end{equation}
\end{proposition}

\begin{proof}
A direct calculation shows that
\begin{equation}\label{Schur-formula0}
\begin{bmatrix} a & b\\ c & d\end{bmatrix} = \begin{bmatrix} 1 & bd^{-1}\\ 0 & 1\end{bmatrix} \begin{bmatrix}a-bd^{-1}c & 0\\ 0 & d\end{bmatrix} \begin{bmatrix} 1 & 0\\ d^{-1}c & 1\end{bmatrix}
\end{equation}
holds. Since the matrices
$$\begin{bmatrix} 1 & bd^{-1}\\ 0 & 1\end{bmatrix} \qquad\text{and}\qquad \begin{bmatrix} 1 & 0\\ d^{-1}c & 1\end{bmatrix}$$are both invertible in $M_{k+l}(\A)$, the stated equivalence of (i) and (ii) immediately follows from \eqref{Schur-formula0}. Moreover, if (i) and (ii) are satisfied, \eqref{Schur-formula0} leads to
$$\begin{bmatrix} a & b\\ c & d\end{bmatrix}^{-1} = \begin{bmatrix} 1 & 0\\ -d^{-1}c & 1\end{bmatrix} \begin{bmatrix} (a-bd^{-1}c)^{-1} & 0\\ 0 & d^{-1}\end{bmatrix} \begin{bmatrix} 1 & -bd^{-1}\\ 0 & 1\end{bmatrix},$$
from which \eqref{Schur-formula} directly follows.
\end{proof}

The following Corollary is a direct consequence of Proposition \ref{Schur} and explains that we need the additional assumption on the block structure made in part (ii) of Definition \ref{linearization}:

\begin{corollary}\label{Cauchy-transforms}
Let $\A$ be a unital and complex algebra and let elements $x_1,\dots,x_n\in\A$ be given. For any polynomial $p\in\C\langle X_1,\dots,X_n\rangle$ that has a linearization
$$L_p = b_0\otimes 1 + b_1 \otimes X_1 + \dots + b_n \otimes X_n \in M_N(\C)\otimes\C\langle X_1,\dots,X_n\rangle$$
with matrices $b_0,\dots,b_n\in M_N(\C)$, the following conditions are equivalent for any complex number $z\in\C$:
\begin{itemize}
 \item[(i)] The operator $z-P$ with $P:=p(x_1,\dots,x_n)$ is invertible in $\A$.
 \item[(ii)] The operator $\Lambda(z) - L_P$ with $\Lambda(z)$ defined as in \eqref{Lambda} and $$L_P := b_0\otimes 1 + b_1 \otimes x_1 + \dots + b_n \otimes x_n \in M_N(\C)\otimes\A$$ is invertible in $M_N(\C)\otimes\A$.
\end{itemize}
Moreover, if (i) and (ii) are fulfilled for some $z\in\C$, we have that
$$\big[(\Lambda(z)-L_P)^{-1}\big]_{1,1} = (z-P)^{-1}.$$
\end{corollary}

\begin{proof}
For all $k\in\N$, we can consider the evaluation homomorphism
$$\Phi^{(k)}:\ M_k(\C)\otimes\C\langle X_1,\dots,X_n\rangle \rightarrow M_k(\C)\otimes\A,\ q\mapsto q(x_1,\dots,x_n).$$
Since invertible elements in $M_k(\C)\otimes\C\langle X_1,\dots,X_n\rangle$ are mapped by $\Phi^{(k)}$ to invertible elements in $M_k(\C)\otimes\A$, we see that $L_P=\Phi^{(N)}(L_p)$ admits due to Definition \ref{linearization} a block decomposition of the form
$$L_P := \begin{bmatrix} 0 & u\\ v & Q\end{bmatrix} \in M_N(\C)\otimes\A$$
where $Q\in M_{N-1}(\C)\otimes\A$ is invertible and $P = - uQ^{-1}v$ holds. Hence, by using Proposition \ref{Schur}, we can deduce that for any $z\in\C$ the matrix
$$\Lambda(z) - L_p = \begin{bmatrix} z & -u\\ -v & -Q\end{bmatrix}$$
is invertible in $M_n(\C)\otimes\A$ if and only if its Schur-complement
$$z+uQ^{-1}v = z-P$$
is invertible in $\A$. This proves the stated equivalence of (i) and (ii). Moreover, we may deduce from \eqref{Schur-formula} that in this case
$$[(\Lambda(z)-L_P)^{-1}]_{1,1} = (z-P)^{-1}$$
holds. This completes the proof.
\end{proof}

Now, it only remains to ensure the existence of linearizations of this kind.

\begin{proposition}
Any polynomial $p\in \C\langle X_1,\dots,X_n\rangle$ admits a linearization $L_p$ in the sense of Definition \ref{linearization}.
\end{proposition}

\begin{proof}
Since each monomial $X_j\in\C\langle X_1,\dots,X_n\rangle$ has the linearization
$$L_{X_j} = \begin{bmatrix} 0 & X_j\\ 1 & -1\end{bmatrix} \in M_2(\C)\otimes\C\langle X_1,\dots,X_n\rangle$$
and all polynomials of the form
$$p:= X_{i_1}X_{i_2}\cdots X_{i_k} \in\C\langle X_1,\dots,X_n\rangle \qquad\text{for $k\geq 2$, $i_1,\dots,i_k\in\{1,\dots,n\}$}$$
have the linearizations
$$L_p = \begin{bmatrix} & & & X_{i_1}\\
& & X_{i_2} & -1\\
& \iddots & \iddots & \\
                       X_{i_k} & -1 & &
        \end{bmatrix}\in M_k(\C)\otimes\C\langle X_1,\dots,X_n\rangle,$$
we just have to observe the following: If the polynomials $p_1,\dots,p_k\in\C\langle X_1,\dots,X_n\rangle$ have linearizations
$$L_{p_j} =  \begin{bmatrix} 0 & u_j\\ v_j & Q_j\end{bmatrix}\in M_{N_j}(\C)\otimes\C\langle X_1,\dots,X_n\rangle$$
for $j=1,\dots,n$, then their sum $p:=p_1+\dots+p_k$ has the linearization
$$L_p = \begin{bmatrix}
& u_1 & \hdots & u_k\\
                v_1 & Q_1 & &    \\
                \vdots & & \ddots &    \\
                v_k & & & Q_k\\
        \end{bmatrix} \in M_N(\C)\otimes\C\langle X_1,\dots,X_n\rangle$$
with $N := (N_1+\dots+N_k)-k+1.$
\end{proof}

It is a nice feature that we can extend this algorithm to preserve selfadjointness. Note that $\C\langle X_1,\dots,X_n\rangle$ becomes a $\star$-algebra by anti-linear extension of
$$(X_{i_1}X_{i_2}\cdots X_{i_k})^\ast := X_{i_k} \cdots X_{i_2} X_{i_1} \qquad\text{for $i_1,\dots,i_k\in\{1,\dots,n\}$}.$$

\begin{corollary}\label{linearization-existence}
Let $p\in \C\langle X_1,\dots,X_n\rangle$ be a selfadjoint polynomial. Then $p$ admits a selfadjoint linearization $L_p$ in the sense of Definition \ref{linearization}.
\end{corollary}

\begin{proof}
Since $p$ is selfadjoint, we have $p=q+q^\ast$ for $q:=\frac{p}{2}$. Let
$$L_q = \begin{bmatrix} 0 & u\\ v & Q\end{bmatrix}\in M_N(\C)\otimes\C\langle X_1,\dots,X_n\rangle \qquad\text{with}\qquad q=-uQ^{-1}v$$
be a linearization of $q$. Then
$$L_p = \begin{bmatrix} 0 & u & v^\ast\\ u^\ast & 0 & Q^\ast\\ v & Q & 0\end{bmatrix} \in M_{2N-1}(\C)\otimes\C\langle X_1,\dots,X_n\rangle$$
gives a selfadjoint linearization of $p$. Indeed, condition (i) of Definition \ref{linearization} is obviously fulfilled and, since
\begin{eqnarray*}
-\begin{bmatrix} u & v^\ast\end{bmatrix}\begin{bmatrix} 0 & Q^\ast\\ Q & 0\end{bmatrix}^{-1} \begin{bmatrix} u^\ast\\ v\end{bmatrix}
&=& -\begin{bmatrix} u & v^\ast\end{bmatrix}\begin{bmatrix} 0 & Q^{-1}\\ (Q^\ast)^{-1} & 0\end{bmatrix} \begin{bmatrix} u^\ast\\ v\end{bmatrix}\\
&=& -u Q^{-1} v- (u Q^{-1} v)^\ast\\
&=& q + q^\ast\\
&=& p,
\end{eqnarray*}
condition (ii) of Definition \ref{linearization} holds as well.
\end{proof}

We conclude with the following corollary, which will enable us to shift $\Lambda(z)$ for $z\in\C^+$ to a point
$$\Lambda_\epsilon(z) := \begin{bmatrix} z & & & \\ & i\epsilon & & \\ & & \ddots & \\ & & & i\epsilon \end{bmatrix}$$
lying inside the domain $\H^+(M_N(\C))$ in order to get access to all analytic tools that are available there.

\begin{corollary}\label{Cauchy-approximation}
Let $(\A,\phi)$ be a $C^\ast$-probability space and let elements $x_1,\dots,x_n\in\A$ be given. For any selfadjoint polynomial $p\in\C\langle X_1,\dots,X_n\rangle$ that has a selfadjoint linearization
$$L_p = b_0\otimes 1 + b_1 \otimes X_1 + \dots + b_n \otimes X_n \in M_N(\C)\otimes\C\langle X_1,\dots,X_n\rangle$$
with matrices $b_0,\dots,b_n\in M_N(\C)_{\sa}$, we put $P:=p(x_1,\dots,x_n)$ and
$$L_P := b_0\otimes 1 + b_1 \otimes x_1 + \dots + b_n \otimes x_n \in M_N(\C)\otimes\A.$$
Then, for each $z\in\C^+$ and all sufficiently small $\epsilon>0$, the operators $z-P\in\A$ and $\Lambda_\epsilon(z)-L_P\in M_N(\C)\otimes\A$ are both invertible and we have
$$\lim_{\epsilon\downarrow 0} \big[\E\big((\Lambda_\epsilon(z)-L_P)^{-1}\big)\big]_{1,1}= G_P(z).$$
Hereby, $\E: M_N(\C)\otimes\A \rightarrow M_N(\C)$ denotes the conditional expectation given by $\E:=\id_{M_N(\C)}\otimes\phi$.
\end{corollary}

\begin{proof}
It is an easy observation that the matrix valued resolvent set of $L_P$, namely the set $\Omega$ of all matrices $b\in M_N(\C)$ for which $b-L_P$ is invertible in $M_N(\C)\otimes\A$, is an open subset of $M_N(\C)$.\\
Since $P$ is selfadjoint, $z-P$ is invertible for all $z\in\C^+$. Hence, it follows from Corollary \ref{Cauchy-transforms} that $\Lambda(z)\in\Omega$ for all $z\in\C^+$. This implies that the function
$$G:\ \Omega \mapsto M_N(\C),\ b\mapsto \E[(b-L_P)^{-1}]$$
is holomorphic (and particularly continuous) in a neighborhood of $\Lambda(z)$ for all $z\in\C^+$. This shows the invertibility of $\Lambda_\epsilon(z) - L_p$ for all $z\in\C^+$ and sufficiently small $\epsilon>0$. By using Corollary \ref{Cauchy-transforms} again, we observe
$$\big[\E[(\Lambda(z)-L_P)^{-1}]\big]_{1,1} = \phi\big(\big[(\Lambda(z)-L_P)^{-1}\big]_{1,1}\big) = \phi\big((z-P)^{-1}\big) = G_P(z).$$
Thus we get that
$$\lim_{\epsilon\downarrow 0} \big[\E\big((\Lambda_\epsilon(z)-L_P)^{-1}\big)\big]_{1,1} = \lim_{\epsilon\downarrow 0} [G(\Lambda_\epsilon(z))]_{1,1} = [G(\Lambda(z))]_{1,1} = G_P(z)$$
holds as claimed.
\end{proof}

\section{The algorithmic solution of Problems 1 and 2}

As already mentioned in the introduction, the subordination result for the operator-valued free additive convolution stated in Section 2 provides in combination with Anderson's version of the linearization trick presented in Section 3 an algorithm to calculate the distribution of any selfadjoint polynomial in selfadjoint elements which are freely independent. The following theorem gives the precise statement.

\begin{theorem}\label{algorithm}
Let $(\A,\phi)$ be a non-commutative $C^\ast$-probability space, $x_1,\dots,x_n\in\A$ selfadjoint elements which are freely independent, and $p\in\C\langle X_1,\dots,X_n\rangle$ a selfadjoint polynomial in $n$ non-commuting variables $X_1,\dots,X_n$. We put $P:=p(x_1,\dots,x_n)$. The following procedure leads to the distribution of $P$.

\textbf{Step 1:} According to Corollary \ref{linearization-existence}, $p$ has a selfadjoint linearization
$$L_p = b_0\otimes 1 + b_1 \otimes X_1 + \dots + b_n \otimes X_n$$
with matrices $b_0,b_1,\dots,b_n\in M_N(\C)_{\sa}$ of dimension $N\in\N$. We put
$$L_P := b_0\otimes 1 + b_1 \otimes x_1 + \dots + b_n \otimes x_n \in M_N(\C)\otimes\A.$$

\vspace{0.2cm}

\textbf{Step 2:} The operators $b_1\otimes x_1, \dots, b_n\otimes x_n$ are freely independent elements in the operator-valued $C^\ast$-probability space $(M_N(\C) \otimes \A, \E)$, where $\E: M_N(\C)\otimes\A \rightarrow M_N(\C)$ denotes the conditional expectation given by $\E:=\id_{M_N(\C)}\otimes\phi$.
Furthermore, for $j=1,\dots,n$, the $M_N(\C)$-valued Cauchy transform $G_{b_j\otimes x_j}$ is completely determined by the scalar-valued Cauchy transforms $G_{x_j}$ via
$$G_{b_j\otimes x_j}(b) = \lim_{\epsilon\downarrow 0} \frac{-1}{\pi}\int_\R (b - t b_j)^{-1} \Im(G_{x_j}(t+i\epsilon))\, dt$$
for all $b\in\H^+(M_N(\C))$.

\vspace{0.2cm}

\textbf{Step 3:} Due to Step 3, we can calculate the Cauchy transform of $$L_P-b_0\otimes 1 =  b_1 \otimes x_1 + \dots + b_n \otimes x_n$$ by using the fixed point iteration for the operator-valued free additive convolution. The Cauchy transform of $L_P$ is then given by
$$G_{L_P}(b) = G_{L_P-b_0\otimes1}(b-b_0) \quad \text{for all $b\in\H^+(M_N(\C))$}.$$

\vspace{0.2cm}

\textbf{Step 4:} Corollary \ref{Cauchy-approximation} tells us that the scalar-valued Cauchy transform $G_P$ of $P$ is determined by
$$G_P(z) = \lim_{\epsilon\downarrow 0} \big[G_{L_P}(\Lambda_\epsilon(z))\big]_{1,1} \qquad\text{for all $z\in\C^+$}.$$
Finally, we obtain the desired distribution of $P$ by applying the Stieltjes inversion formula.

\end{theorem}

Based on the results of the previous sections, there is almost nothing left to prove here. 
Only the two statements in Step 2 might require an explanation. Firstly, the fact that scalar-valued freeness between the entries of matrices $A_1,\dots,A_n\in M_N(\A)$ implies freeness over $M_N(\C)$ of the matrices themselves is a direct consequence of the definition of freeness. Secondly,  
the validity of the formula  for the operator-valued Cauchy transform of $b_j\otimes x_j$ in terms of the scalar valued Cauchy transform of $G_{x_j}$ can be seen as follows: 
Let $\mu_j$ be the distribution of $x_j$. It is an easy observation that for all $b\in \H^+(M_N(\C))$ the formula
$$\E\big[(b - b_j\otimes x_j)^{-1}\big] = \int_\R (b- t b_j)^{-1}\, d\mu_j(t)$$
holds. Hereby, the integral on the right hand side is defined entrywise, which corresponds in this case to the definition of the well-known Bochner integral. By using the Stieltjes inversion theorem, this leads us to
$$\E\big[(b - b_j\otimes x_j)^{-1}\big] = \lim_{\epsilon\downarrow 0} \frac{-1}{\pi}\int_\R (b - t b_j)^{-1} \Im(G_{\mu_j}(t+i\epsilon))\, dt.$$
Particularly, this means that knowing the Cauchy transform $G_{x_j}$ of $x_j$ suffices to calculate the matrix-valued Cauchy transform of $b_j\otimes x_j$.

\section{Examples}
In this section we will present a few examples of different polynomials to 
show the power and universality of our results. We will in all those cases compare the distribution arising from our algorithm with histograms of eigenvalue distributions of corresponding random matrices. For the random matrices we will chose either Gaussian or Wishart random matrices; their limit distribution is almost surely given by the semicircle law and the Marchenko-Pastur (aka free Poisson) distribution, respectively, and independent choices of such matrices will almost surely be asymptotically free (with respect to the normalized trace on the matrices); see, for example, \cite{VDN,HP,NSbook}. 

\begin{example}[The anticommutator]
We consider the non-commutative polynomial $p\in\C\langle X_1,X_2\rangle$ given by
$$p(X_1,X_2) = X_1X_2 + X_2X_1.$$
It is easy to check that
$$L_p = \begin{bmatrix} 0 & X_1 & X_2\\ X_1 & 0 & -1\\ X_2 & -1 & 0\end{bmatrix}$$
is a selfadjoint linearization of $p$ in the sense of Definition \ref{linearization}.

Now, let $s_1,s_2$ be free semicircular elements in a non-commutative $C^\ast$-probability space $(\A,\phi)$. Based on the algorithm of Theorem \ref{algorithm}, we can calculate the distribution of the anticommutator
$p(s_1,s_2)=s_1s_2+s_2s_1$.

Consider, on the other hand, two sequences $(X^{(n)}_1)_{n\in\N}$ and $(X^{(n)}_2)_{n\in\N}$ of independent standard Gaussian random matrices. For each $n\in\N$, we consider the empirical eigenvalue distribution of the anticommutator
$$p(X^{(n)}_1,X^{(n)}_2)=X^{(n)}_1X^{(n)}_2 + X^{(n)}_2X^{(n)}_1.$$
By the asymptotic freeness results of Voiculescu,
the eigenvalue distribution of those matrix anticommutators converges, as $n$ tends to infinity,  to the distribution of the anticommutator $s_1s_2+s_2s_1$. Fig. \ref{fig:example1} compares the eigenvalue distribution of $p(X^{(n)}_1,X^{(n)}_2)$, for $n=4000$, with the result of our algorithm for the distribution of $p(s_1,s_2)$.

\begin{figure}
\epsfig{width=0.7\columnwidth, file=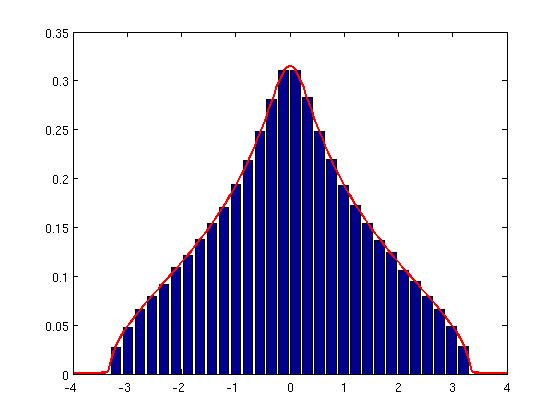}
\caption{Comparison of the distribution of $p(s_1,s_2)$, according to our algorithm, with the histogram of eigenvalues for $p(X^{(n)}_1,X^{(n)}_2)$, for $n=4000$; here $p$ is the free anticommutator}
\label{fig:example1}
\end{figure}
\end{example}

\begin{example}[Perturbated anticommutator]

In the same way, we can deal with the following variation of the anticommutator:
$$p(X_1,X_2) = X_1X_2+X_2X_1+X_1^2$$
It is easy to check that
$$L_p = \begin{bmatrix} 0 & X_1 & \frac{1}{2} X_1 + X_2\\ X_1 & 0 & -1\\ \frac{1}{2} X_1 + X_2 & -1 & 0 \end{bmatrix}$$
is a selfadjoint linearization of $p$ in the sense of Definition \ref{linearization}.
Fig. \ref{fig:example2} compares the result of our algorithm for $p(s,w)$ - where $s,w$ are free 
and $s$ is a semicircular and $w$ a free Poisson element - with random matrix simulations of the eigenvalues 
of $p(X,W)$, where $X$ is  a $4000\times 4000$ Gaussian 
random matrices and $W$ is a $4000\times 4000$ Wishart random matrix, both chosen independently.

\begin{figure}
\epsfig{width=0.7\columnwidth, file=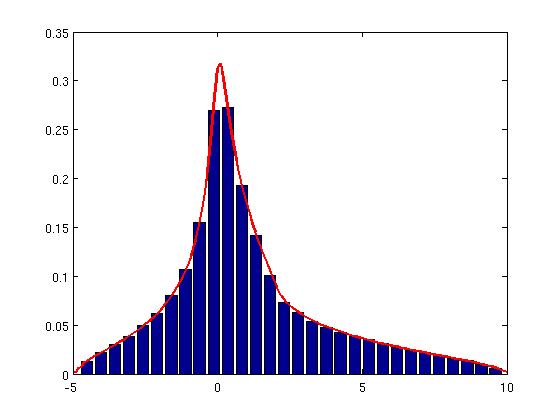}
\caption{Normalized eigenvalue histogram for one realization of 
$p(X_1,X_2) = X_1X_2+X_2X_1+X_1^2$ for one independent realization of a 
$4000\times 4000$ Gaussian random matrix $X_1$ and a $4000\times 4000$ Wishart matrix $X_2$, compared 
to the density of the distribution of $p(s,w)$ for a semicircular element $s$ and a free Poisson element $w$, which are free.}
\label{fig:example2}
\end{figure}
\end{example}

\begin{example}
We consider the polynomial
$$p(X_1,X_2,X_3) = X_1X_2X_1 + X_2X_3X_2 + X_3X_1X_3.$$
It is not hard to see that
$$L_p = \begin{bmatrix} 0 & 0 & X_1 & 0 & X_2 & 0 & X_3\\
                        0 & X_2 & -1 & 0 & 0 & 0 & 0\\
                        X_1 & -1 & 0 & 0 & 0 & 0 & 0\\
                        0 & 0 & 0 & X_3 & -1 & 0 & 0\\
                        X_2 & 0 & 0 & -1 & 0 & 0 & 0\\
                        0 & 0 & 0 & 0 & 0 & X_1 & -1\\
                        X_3 & 0 & 0 & 0 & 0 & -1 & 0
         \end{bmatrix}$$
gives a selfadjoint linearization of $p$ in the sense of Definition \ref{linearization}. 
If we consider independent Gaussian random matrices
$X_1^{(n)}$ and $X_2^{(n)}$ and Wishart random matrices $X_3^{(n)}$, the limiting eigenvalue distribution 
is given by the distribution of $p(s_1,s_2,w)$, where $s_1,s_2,w$ are free elements in some $C^\ast$-probability 
space $(\A,\phi)$, such that $s_1$ and $s_2$ are semicircular elements and $w$ is a 
free Poisson. Fig. \ref{fig:example4-1} compares the result of our algorithm for $p(s_1,s_2,w)$  with the eigenvalues
of a random matrix simulation for $n=4000$.

\begin{figure}
\epsfig{width=0.7\columnwidth, file=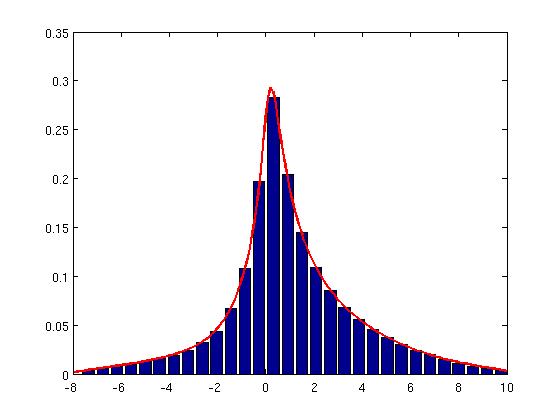}
\caption{Normalized eigenvalue histogram for one realization of $p(X_1,X_2,X_3) = X_1X_2X_1 + X_2X_3X_2 + X_3X_1X_3$ for one independent realization of two $4000\times 4000$ Gaussian matrices and one $4000\times 4000$ Wishart random matrix, compared to the density of the distribution of $p(s_1,s_2,w)$ for free semicirculars $s_1,s_2$ and a free Poisson element $w$}
\label{fig:example4-1}
\end{figure}
\end{example}

\end{document}